\documentclass[11pt,reqno]{amsart}
\usepackage{amssymb,mathrsfs,graphicx}
\usepackage{ifthen}
\usepackage{colortbl}
\usepackage{tikz}
\usepackage{cite}                            

\newlength\imagewidth
\newlength\imagescale

\definecolor{black}{rgb}{0.0, 0.0, 0.0}
\definecolor{red}{rgb}{1.0, 0.0, 0.0}
\provideboolean{shownotes} 
\setboolean{shownotes}{true} 
\newcommand{\margnote}[1]{
\ifthenelse{\boolean{shownotes}}%
{\marginpar{\raggedright\tiny\texttt{#1}}}%
{}%
}
\newcommand{\hole}[1]{
\ifthenelse{\boolean{shownotes}}%
{\begin{center} \fbox{ \rule {.25cm}{0cm} \rule[-.1cm]{0cm}{.4cm}
\parbox{.85\textwidth}{\begin{center} \texttt{#1}\end{center}} \rule
{.25cm}{0cm}}\end{center}} {} }
\renewcommand{\theta}{\vartheta}

\title[Structure preserving schemes for the continuum Kuramoto model]{Structure preserving schemes for the continuum Kuramoto model: phase transitions}

\author[Carrillo]{Jos\'{e} A. Carrillo}
\address[Jos\'{e} A. Carrillo]{\newline Department of Mathematics
    \newline Imperial College London, London SW7 2AZ, United Kingdom}
\email{carrillo@imperial.ac.uk}

\author[Choi]{Young-Pil Choi}
\address[Young-Pil Choi]{\newline Department of Mathematics and Institute of Applied Mathematics
\newline Inha University, Incheon, 402-751, Republic of Korea}
\email{ypchoi@inha.ac.kr}

\author[Pareschi]{Lorenzo Pareschi}
\address[Lorenzo Pareschi]{\newline Department of Mathematics and Computer Science
\newline University of Ferrara, Via N. Machiavelli 35, 44121, Ferrara, Italy}
\email{lorenzo.pareschi@unife.it}

\numberwithin{equation}{section}

\newtheorem{lemma}{Lemma}[section]

\newtheorem{proposition}{Proposition}[section]
\newtheorem{remark}{Remark}[section]
\newtheorem{definition}{Definition}[section]

\newcommand{\R}{\mathbb R}
\newcommand{\N}{\mathbb N}
\newcommand{\Z}{\mathbb Z}
\newcommand{\bs}{\mathbb S}

\newcommand{\om}{\omega}

\newcommand{\T}{\mathbb T}
\newcommand{\mc}{\mathcal C}
\newcommand{\bq}{\begin{equation}}
\newcommand{\eq}{\end{equation}}

\newcommand{\lt}{\left}
\newcommand{\rt}{\right}
\newcommand{\lal}{\langle}
\newcommand{\ral}{\rangle}
\newcommand{\pa}{\partial}
\newcommand{\mm}{\mathcal{M}}
\newcommand{\ms}{\mathcal{S}}

\begin{document}
\allowdisplaybreaks

\date{\today}

\subjclass[]{}

\begin{abstract} The construction of numerical schemes for the Kuramoto model is challenging due to the structural properties of the system which are essential in order to capture the correct physical behavior, like the description of stationary states and phase transitions. Additional difficulties are represented by the high dimensionality of the problem in presence of multiple frequencies. In this paper, we develop numerical methods which are capable to preserve these structural properties of the Kuramoto equation in the presence of diffusion and to solve efficiently the multiple frequencies case. The novel schemes are then used to numerically investigate the phase transitions in the case of identical and non identical oscillators. 
\end{abstract}

\maketitle \centerline{\date}


%
%
%
%
\section{Introduction} 

Synchronization phenomena of large populations of weakly coupled oscillators are very common in natural systems, and it has been extensively studied in various scientific communities such as physics, biology, sociology, etc. \cite{ABPRS, Erm, Ku, Str}. Synchronization arises due to the adjustment of rhythms of self-sustained periodic oscillators weakly connected \cite{ABPRS, PRK}, and its rigorous mathematical treatment is pioneered by Winfree \cite{Win} and Kuramoto \cite{Ku}. In \cite{Ku, Win}, phase models for large weakly coupled oscillator systems were introduced, and the synchronized behavior of complex biological systems was shown to emerge from the competing mechanisms of intrinsic randomness and sinusoidal couplings. Since then, the Kuramoto model becomes a prototype model for synchronization phenomena and various extensions have been extensively explored in various scientific communities such as applied mathematics, engineering, control theory, physics, neuroscience, biology, and so on \cite{DB, PRK, Str}.

Given an ensemble of sinusoidally coupled nonlinear oscillators, which can be visualized as active rotors on the unit circle $\bs^1$, let $z_j = e^{i\theta_j}$ be the position of the $j$-th rotor. Then, the dynamics of $z_j$ is completely determined by that of its phase $\theta_j$. Let us denote the phase and frequency of the $j$-th oscillator by $\theta_j$ and $\dot\theta_j$, respectively. Then, the phase dynamics of Kuramoto oscillators are governed by the following first-order ODE system \cite{Ku}:
\bq\label{pku}
\dot\theta_i = \omega_i - \frac KN \sum_{j=1}^N \sin(\theta_i - \theta_j), \quad i = 1,\cdots, N, \quad t > 0,
\eq
subject to initial data $\theta_i(0) =:\theta_{i0}$, $i = 1,\cdots, N$,
where $K$ is the uniform positive coupling strength, and $\omega_i$ denotes the natural phase-velocity (frequency) and is assumed to be a random variable extracted from the given distribution $g  = g(\omega)$ satisfying
\[
\int_\R g(\om) \,d\om = 1. 
\]
We notice that the first term and second term in the right hand side of the equation \eqref{pku} represent the intrinsic randomness and the nonlinear attraction-repulsion coupling, respectively.

The system \eqref{pku} has been extensively studied, and it still remains a popular subject in nonlinear dynamics and statistical physics. We refer the reader to \cite{ABPRS} and references therein for general survey of the Kuramoto model and its variants. In \cite{Ku}, Kuramoto first observed a continuous phase transition in the continuum Kuramoto model ($N \to \infty$, see \eqref{kku} below) with a symmetric distribution function $g(\omega)$ by introducing an order parameter $r^\infty$ which measures the degree of the phase synchronization in the mean-field limit. More precisely, the order parameter is given by
\[
r^N(t)e^{i \varphi_N(t)} := \frac1N \sum_{j=1}^N e^{i \theta_j(t)}, \quad r^\infty := \lim_{t \to \infty}\lim_{N \to \infty} r^N(t).
\]
In fact, Kuramoto showed that the order parameter $r^\infty$ as a function of the coupling strength $K$ changes from zero (disordered state) to a non-zero value (ordered state) when the coupling strength $K$ exceeds a critical value $K_c := 2/(\pi g(0))$, i.e., $r^\infty(K) = 0$ (incoherent state) for $K \in [0,K_c)$, $1 \geq r^\infty(K) > 0$ (coherent state) for $K > K_c$, and $r^\infty(K)$ increases with $K$. Later, it is also observed that the continuum Kuramoto model can exhibit continuous or  discontinuous phase transition by taking into account different types of natural frequency distribution functions \cite{BU1, BU2, Paz}. Here, continuous phase transitions refer to the continuity at $K=K_c$ of the order parameter $r^\infty(K)$. It is known to be continuous for the Gaussian distributed in frequency oscillators \cite{Str} and discontinuous for the uniformly distributed in frequency oscillators \cite{Paz}.

As the number of oscillators goes to infinity $(N \to \infty)$, a continuum description of the system \eqref{pku} can be rigorously derived by employing by now standard mean-field limit techniques for the Vlasov equation \cite{Neu,Lan,CCR,CCHKK}. Let $\rho = \rho(\theta, \om, t)$ be the probability density function of Kuramoto oscillators in
$\theta \in \T := \R/(2\pi \Z)$ with a natural frequency
$\om$ extracted from a distribution function $g = g(\om)$ at time $t$. Then the continuum  Kuramoto
equation is given by
\begin{align}
\begin{aligned} \label{kku}
\displaystyle &\partial_t \rho + \partial_{\theta} (u[\rho] \rho) = 0, \qquad (\theta, \omega) \in \T \times \R,~~t > 0, \\
& u[\rho](\theta, \omega, t) = \omega - K \int_{\T \times \R} \sin(\theta-\theta_*) \rho(\theta_*, \omega, t) g(\omega)\,d\theta_* d\omega,
\end{aligned}
\end{align}
subject to the initial data:
\begin{equation} \label{ini_kku}
\rho(\theta, \omega, 0) =: \rho_0(\theta, \omega), \quad  \int_{\T} \rho_0(\theta,\omega) \,d\theta = 1.
\end{equation}
The order parameter $r$ and the average phase $\varphi$ associated to \eqref{kku} are given as
$$
r(t) e^{i\varphi(t)} = \int_{\T \times \R} e^{i\theta}\rho(\theta, \omega, t)g(\omega)\,d\theta d\omega,
$$
leading to
\bq\label{eq_ko}
r(t) = \int_{\T \times \R} \cos(\theta - \varphi(t))\rho(\theta,\omega,t)g(\omega)\,d\theta d\omega.
\eq
For the continuum equation \eqref{kku}, global existence and uniqueness of measure-valued solutions are studied in \cite{Lan,CCHKK} and important qualitative properties such as the Landau damping towards the incoherent stationary states were analyzed in \cite{FGG}.

A very relevant issue from the application viewpoint is how stable these stationary states and phase transitions are by adding noise to the system \cite{Ku,Saka}. The mean-field limit equation associated to \eqref{pku} with standard Gaussian noise of strength $\sqrt{2D}$, with $D>0$, is
\begin{align}
\begin{aligned} \label{kku_n}
\displaystyle &\partial_t \rho + \partial_{\theta} (u[\rho] \rho) = D\pa^2_\theta\rho, \qquad (\theta, \omega) \in \T \times \R,~~t > 0, \\
& u[\rho](\theta, \omega, t) = \omega - K \int_{\T \times \R} \sin(\theta-\theta_*) \rho(\theta_*, \omega, t) g(\omega)\,d\theta_* d\omega,
\end{aligned}
\end{align}
subject to the initial data \eqref{ini_kku}. Phase transitions have also been found in terms of the order parameter $r^\infty$ as a function of $K$ for $D>0$. In fact, it was proven in \cite{Saka} that for identical oscillators there is a continuous phase transition. This phase transition is also continuous for Gaussian and uniformly coupled noisy oscillators \cite{ABPRS} and discontinuous for bimodal distributions \cite{BNS,Craw94}. Stability of the coherent and incoherent states and the asymptotic dynamics for the Kuramoto model with noise \eqref{kku_n} were analyzed in \cite{GLP,GPP,HX1,HX2}. Noise-driven phase transitions in interacting particle systems are a hot topic of research from the numerical and theoretical viewpoints, see \cite{BM,BarDeg,BD00,BCCD,BDZ,BCDPZ,Craw94,GP} and the references therein.

In the present work, we focus on the construction of effective numerical schemes for the continuum Kuramoto system with diffusion obtained in the limit of a very large number of oscillators and in presence of noise. The numerical solution of such system is challenging due to the high dimensionality of the problem and the intrinsic structural properties of the system which are essential in order to capture the correct physical behavior. The literature dealing with the study of numerical schemes to the continuum Kuramoto model \eqref{kku} is mainly focused on deterministic spectral and stochastic methods. We refer to \cite[Section VI. B]{ABPRS} and the references therein for general discussions on the stochastic and deterministic numerical methods used in Kuramoto models, see also \cite{PRK,GCR1,GCR2,BM} for Monte Carlo related methods and \cite{PR97,AB98,APS01} for spectral methods.

Structure preserving numerical schemes for mean-field type kinetic equations have been previously constructed in \cite{ABPRS, BCDS, BD,BF, CCH, ChCo, PZ, SG}. We refer also to \cite{Gosse} for related approaches for general systems of balance laws and to \cite{DP15} for an introduction to numerical methods for kinetic equations. Here, following the approach in \cite{BD, ChCo, PZ} we introduce a discretization of the phase variable such that nonnegativity of the solution, physical conservations, asymptotic behavior and free energy dissipation are preserved at a discrete level for identical oscillators. This approach is then coupled with a suitable collocation method for the frequency variable based on orthogonal polynomials with respect to the given frequency distribution. Similar collocation strategies have been used for kinetic equations in the case of linear transport and semiconductors \cite{JP1, LM, RSZ}. 

The rest of the manuscript is organized as follows. First, in Section 2 we recall some basic properties of the continuum Kuramoto model, in particular concerning some relevant existence theory and some useful estimates underpinning our numerical strategy. We also discuss the asymptotic behavior of the Kuramoto model with noise, stationary states and related free energy estimates in the case of identical oscillators. The discretization of Kuramoto systems is discussed in Section 3. Structure preserving schemes are developed in combination with a collocation method for the frequency variable. The asymptotic behavior of the schemes as well as the other relevant physical properties are discussed. In Section 4 we present several numerical tests, with a particular emphasis on the study of phase transitions. {We will also compare our structure preserving methods to Monte Carlo and spectral methods.} Future research directions and conclusions are reported in the last Section.

%
%
%
%
\section{The continuum Kuramoto model}\label{sec_det}

\subsection{Stability of the Kuramoto model with respect to the frequency distribution}

In this subsection, we briefly review some relevant issues related to the well-posedness theory and several useful estimates for the continuum Kuramoto equation \eqref{kku} and its version with noise \eqref{kku_n}. We refer to \cite{Lan,CCR,CCHKK} for further details. We first provide {\it a priori} estimates for the equation \eqref{kku}.

\begin{lemma}\label{lem_bas} Let $\rho$ be a smooth solution to \eqref{kku}-\eqref{ini_kku}. Then we have
$$\begin{aligned}
&(i) \quad \frac{d}{dt}\int_{\T} \rho(\theta,\omega,t)\,d\theta = 0 \quad \mbox{i.e.,} \quad \int_\T \rho(\theta,\omega,t)\,d\theta = \int_\T \rho_0(\theta,\omega)\,d\theta = 1 \quad \mbox{for} \quad w \in \R,\cr
&(ii) \quad \frac{d}{dt}\int_{\T \times \R} \rho(\theta,\omega,t) \log (\rho(\theta,\omega,t))g(\omega)\,d\theta d\omega = K r^2(t),
\end{aligned}$$
for $t \geq 0$, where $r$ is given in \eqref{eq_ko}.
\end{lemma}
\begin{proof} The proof of $(i)$ is straightforward. For the identity $(ii)$, by using the definition of $r$ in \eqref{eq_ko} we rewrite $u[\rho]$ as
\bq\label{rew_u}
u[\rho](\theta,\omega,t) = \omega - Kr \sin(\theta - \varphi(t)).
\eq
Direct computations yield
\begin{align*}
\frac{d}{dt}\int_{\T \times \R} \rho(\theta,\omega,t) \log (\rho(\theta,\omega,t))g(\omega)\,d\theta d\omega = &\, -\int_{\T \times \R} g(\omega) \pa_\theta(u[\rho]\rho)\log \rho \,d\theta d\omega \\
= &\, \int_{\T \times \R} g(\omega)u[\rho]\pa_\theta \rho\,d\theta d\omega \\
= &\, -\int_{\T \times \R} (\pa_\theta u[\rho]) \rho(\theta,\omega) g(\omega)\,d\theta d\omega\\
= &\, Kr\int_\T \cos(\theta-\varphi(t))\rho(\theta, \omega)g(\omega)\,d\theta d\omega = Kr^2\,,
\end{align*}
completing the proof.
\end{proof}

We next discuss global existence and uniqueness of measure-valued solutions. Let us denote by $\mm(\T \times \R)$ the set of nonnegative Radon measures on $\T \times \R$, which can be regarded as nonnegative bounded linear functionals on $\mc_0(\T \times \R)$. Then our notion of measure-valued solutions to \eqref{kku} is defined as follows.

\begin{definition} For $T \in [0,\infty)$, let $\mu \in \mc_w(0,T\,;\mm(\T \times \R))$ be a measure-valued solution to \eqref{kku} with an initial measure $\mu_0 \in \mm(\T \times \R)$ if and only if $\mu$ satisfies the following conditions:
\begin{itemize}
\item $\mu$ is weakly continuous, i.e.,
\[
\lal \mu_t, h\ral \mbox{ is a continuous as a function of $t$}, \quad \forall h \in \mc_0(\T \times \R),
\]
where 
\[
\lal \mu_t, h\ral := \int_{\T \times \R}h(\theta,\omega)\, \mu_t(d\theta,d\omega).
\]
\item For $h \in \mc^1_0(\T \times \R \times [0,T))$, $\mu$ satisfies the following integral equation:
\[
\lal \mu_t, h(\cdot,\cdot,t)\ral - \lal \mu_0, h(\cdot,\cdot,0)\ral = \int_0^t \lal \mu_s, \pa_s h + u[\mu]\pa_\theta h\ral\,ds,
\] 
where $u[\mu]$ is given by
\[
u[\mu](\theta,\omega,t) := \omega - K\int_{\T \times \R} \sin (\theta- \theta_*)g(\omega)\,\mu_t(d\theta_*, d\omega).
\]
\end{itemize}
\end{definition}
We also introduce the definition of the bounded Lipschitz distance. 
\begin{definition}Let $\mu,\nu \in \mm (\T \times \R)$ be two Radon measures. Then the bounded Lipschitz distance $d(\mu,\nu)$ between $\mu$ and $\nu$ is given by
\[
d(\mu,\nu) := \sup_{h \in \ms} \lt| \lal \mu, h\ral - \lal \nu, h\ral \rt|,
\]
where the admissible set $\ms$ of test functions is defined as 
\[
\ms := \lt\{ h : \T \times \R \to \R : \max\lt\{\|h\|_{L^\infty}, Lip(h)\rt\} \leq1  \rt\}
\] 
with 
\[
Lip(h) := \sup_{(\theta,\omega) \neq (\theta_*,\omega_*)} \frac{\lt|h((\theta,\omega)) - h((\theta_*,\omega_*))\rt|}{\lt|(\theta,\omega) - (\theta_*,\omega_*)\rt|}.
\]
\end{definition}
We now present the results on the global existence and stability of measure-valued solutions to \eqref{kku}. We refer the reader to \cite{Lan,CCR,CCHKK} for details of the proof.

\begin{proposition}\label{prop_exs} For any $\mu_0 \in \mm(\T \times \R)$, \eqref{kku} has a unique solution $\mu \in \mc_w(0,T\,; \mm(\T \times \R))$ with the initial data $\mu_0$. Furthermore, if $\mu$ and $\nu$ are two such solutions to \eqref{kku}, then there exists a constant $C > 0$ depending on $T$ such that
\[
d(\mu_t, \nu_t) \leq C d(\mu_0,\nu_0) \quad \mbox{for} \quad t \in [0,T).
\]
\end{proposition}

\begin{remark}\label{rmk_push} One can also characterize the solution $\mu$ of the initial-value problem for \eqref{kku} as the push-forward of the initial data $\mu_0$ through the flow map generated by $u[\mu]$, i.e., for any $h \in \mc^1_c(\T \times \R)$ and $t,s \geq 0$
\[
\int_{\T \times \R} h(\theta,\omega)\,\mu_t(d\theta,d\omega) = \int_{\T \times \R} h(\Theta(0;t,\theta,\omega), \omega)\,\mu_0(d\theta,d\omega),
\]
where $\Theta$ satisfies
\bq\label{eq_Theta}
\frac{d}{dt} \Theta(t;s,\theta,\omega) = u[\mu](\Theta(t;s,\theta,\omega), \omega, t) \quad \mbox{with} \quad \Theta(s;s,\theta,\omega) = \theta.
\eq
Note that the characteristic system \eqref{eq_Theta} is well defined since the velocity field $u[\mu]$ is globally Lipschitz in $(\theta,\omega)$ and continuous on $t$ even for general measures $\mu \in \mc_w(0,T\,;\mm(\T \times \R))$.
\end{remark}

\begin{remark}\label{rmk_ap}
As a simple extension of Proposition \ref{prop_exs}, see \cite{CCHKK}, we also obtain that the solution $\mu_t$ can be approximated as a sum of Dirac measures of the following form
\[
\mu_t^N = \frac1N \sum_{i=1}^N \delta_{\theta_i(t)} \otimes \delta_{\omega_i},
\]
i.e., $d(\mu_t, \mu_t^N) \to 0$ as $N \to \infty$. This can be accomplished by choosing a particle approximation of the initial data. One procedure to construct particle approximations is the following. Let us choose the smallest square containing the support of the initial data $\mu_0$, i.e., supp$(\mu_0) \subset R$. For a given $n$, we divide the square $R$ into $n^2$ subsquares $R_i$, i.e.,
\[
R = \bigcup_{i=1}^{n^2} R_i.
\]
Let $\theta_i^0,\omega_i$ be the center of $R_i$ for $i= 1,\cdots, n^2$. Then we construct the initial approximation $\mu_0^n$ as 
\[
\mu_0^n = \sum_{i=1}^{n^2} m_i \delta_{\theta_i^0} \otimes \delta_{\omega_i} \quad \mbox{with} \quad m_i = \int_{R_i} \mu_0(\theta,\omega)\,d\theta d\omega.
\]
Then we can easily show that $d(\mu_0^n,\mu_0) \leq C/n \to 0$ as $n \to \infty$.
\end{remark}

From now on, we assume that the initial measure is a smooth absolutely continuous with respect to Lebesgue density with connected support. This assumption can be removed by the stability property in Proposition \ref{prop_exs}. We next show stability of the order parameter $r$ with respect to approximated frequency distributions and initial data by again using the stability estimate presented in Proposition \ref{prop_exs}.

\begin{proposition}\label{prop_conv_r2}
Let $\mu$ be the measure-valued solution to the equation \eqref{kku}-\eqref{ini_kku} in the time interval $[0,T]$ with $g \in \ms$ and the initial data $\mu_0 \in \mm(\T \times \R)$. Let us define $r^n$ as
\[
r^n(t)e^{i \varphi^n(t)} = \int_{\T \times \R} e^{i\theta}\rho^n(\theta,\omega,t)g^n(\omega)\,d\theta d\omega,
\]
where $\rho^n$ is the measure-valued solution to \eqref{kku} with the initial data $\rho_0^n \in \ms$ satisfying $d(\rho_0^n,\mu_0) \to 0 $. Then for any $g^n \in \mm(\R)$ satisfying $d(g^n,g) \to 0$ as $n \to \infty$, we have
\[
\lim_{n\to\infty} r^n(t) = r(t) \quad \mbox{for} \quad t \in [0,T].
\]
\end{proposition}

\begin{proof} We first estimate the $1$-Wasserstein distance between $\rho^n$ and $\mu$. Observe that $\rho^n(t) \in \ms$ for all $t\geq 0$ by the push-forward characterization in Remark \ref{rmk_push}. Moreover, it also follows from Remark \ref{rmk_push} that
$$\begin{aligned}
&\lt|\int_{\T \times \R} h(\theta,\omega)\rho^n(\theta,\omega,t)\,d\theta d\omega - \int_{\T \times \R} h(\theta,\omega)\,\mu_t(d\theta,d\omega)\rt|\cr
&\quad =\lt|\int_{\T \times \R} h(\Theta^n(0;t,\theta,\omega),\omega)\rho^n_0(\theta,\omega)\,d\theta d\omega - \int_{\T \times \R} h(\Theta(0;t,\theta,\omega),\omega)\,\mu_0(d\theta,d\omega)\rt|,
\end{aligned}$$
for any $h \in \mc^1_c(\T \times \R)$ and $t,s \geq 0$, where $\Theta^n$ satisfies
$$\begin{aligned}
\frac{d}{dt} \Theta^n(t;s,\theta,\omega) &= u^n[\rho^n](\Theta^n(t;s,\theta,\omega), \omega, t) \cr
&= \omega - K \int_{\T \times \R} \sin(\Theta^n(t;s,\theta,\omega)-\theta_*) \rho^n(\theta_*, \omega, t) g^n(\omega)\,d\theta_* d\omega,
\end{aligned}$$
with $\Theta^n(s;s,\theta,\omega) = \theta$ for all $n \geq 1$ and $\Theta$ satisfies \eqref{eq_Theta}. For notational simplicity, we denote $\Theta^n(t;0,\theta,\omega) = \Theta^n(t)$ and $\Theta(t;0,\theta,\omega) = \Theta(t)$. Then a straightforward computation yields
$$
\begin{aligned}
\frac{d}{dt}|\Theta^n(t) - \Theta(t)| &\leq K\lt|\int_{\T \times \R} \sin(\Theta^n(t) - \theta_*) \rho^n(\theta_*,\omega,t) \,d\theta_* (g^n - g)(dw) \rt|\cr
&\quad + K\lt|\int_{\T \times \R} \sin(\Theta^n(t) - \theta_*)  g(w)\,(\rho^n_t - \mu_t)(d\theta_*,d\omega) \rt|\cr
&\quad + K\lt|\int_{\T \times \R} \lt(\sin(\Theta^n(t) - \theta_*)  - \sin(\Theta(t) - \theta_*)\rt) g(w)\,\mu_t(d\theta_*,d\omega)  \rt|\cr
&\leq Kd(g^n,g) + Kd(\rho^n_t,\mu_t) + K|\Theta^n(t) - \Theta(t)|. 
\end{aligned}
$$
Thus we conclude
\[
|\Theta^n(t) - \Theta(t)| \leq C(K,T)\lt( d(g^n,g) +  \int_0^t d(\rho^n_s,\mu_s)\,ds\rt),
\]
where $C>0$ is independent of $n$. This together with the boundedness of $\Theta^n$ and $\Theta$ in $[0,T]$ gives
$$\begin{aligned}
&\lt|\int_{\T \times \R} h(\Theta^n(0;t,\theta,\omega),\omega)\rho^n_0(\theta,\omega)\,d\theta d\omega - \int_{\T \times \R} h(\Theta(0;t,\theta,\omega),\omega)\,\mu_0(d\theta d\omega)\rt|\cr
&\quad\qquad\leq C\lt( d(g^n,g) +  \int_0^t d(\rho^n_s,\mu_s)\,ds\rt) + Cd(\rho^n_0, \mu_0).
\end{aligned}$$
Hence we have
\[
d(\rho^n_t, \mu_t) \leq C\lt( d(g^n,g) +  \int_0^t d(\rho^n_s,\mu_s)\,ds\rt) + Cd(\rho^n_0, \mu_0), 
\]
i.e.,
\[
d(\rho^n_t, \mu_t) \leq Cd(g^n,g) + Cd(\rho^n_0, \mu_0),
\]
where $C>0$ is independent of $n$. We now show the strong convergence of $r^n$ to $r$. For this, we use the following identities
\[
r^n(t) = \lt|\int_{\T \times \R} e^{i\theta}\rho^n(\theta,\omega,t)g^n(\omega)\,d\theta d\omega\rt| \quad \mbox{and} \quad r(t) = \lt|\int_{\T \times \R} e^{i\theta}g(\omega)\,\mu_t(d\theta,d\omega)\rt|,
\]
to obtain
$$\begin{aligned}
|r^n(t) - r(t)| 
&\leq \lt|\int_{\T \times \R}\cos\theta \lt(\rho^n_t g^n - \mu_t g\rt)(d\theta,d\omega)\rt| + \lt|\int_{\T \times \R}\sin\theta \lt(\rho^n_t g^n - \mu_t g\rt)(d\theta,d\omega)\rt|\cr
&=:I_1 + I_2.
\end{aligned}
$$
Here $I_1$ can be estimated as
$$\begin{aligned}
I_1 &\leq \lt|\int_{\T \times \R}\cos\theta \left[\lt(\rho^n_t - \mu_t\rt)g\right](d\theta,d\omega) \rt| + \lt|\int_{\T \times \R}\cos\theta \left[\lt(g^n - g\rt)\rho^n_t\right](d\theta,d\omega) \rt|\cr
&\leq d(\rho^n_t,\mu_t) + d(g^n,g),
\end{aligned}$$ 
due to $\cos\theta \,g(\omega), \cos\theta\,\rho^n_t(\theta,\omega) \in \ms$. Similarly, we also find
$I_2 \leq d(\rho^n_t,\mu_t) + d(g^n,g)$, concluding that
$$\begin{aligned}
|r^n(t) - r(t)| &\leq I_1 + I_2 \leq 2d(\rho^n_t,\mu_t) + 2d(g^n,g) \cr
&\leq Cd(g^n,g) + Cd(\rho^n_0, \mu_0) \to 0 \quad \mbox{as} \quad n \to \infty.
\end{aligned}$$
This completes the proof.
\end{proof}

\begin{remark}
A similar result to Proposition \ref{prop_conv_r2} can be obtained if we approximate the initial data $\mu_0 \in \mm(\T \times \R)$ and the frequency distribution $g^n(\omega) \in L^\infty(\R)$ as follows: $\mu_0^n \in \mm(\T \times \R)$ and $g^n(\omega) \in \ms$ satisfying $d(\mu_0^n, \mu_0) \to 0$ and $\|g^n(\cdot) - g(\cdot)\|_{L^\infty(\R)} \to 0$ as $n \to \infty$.
\end{remark}

\begin{remark}\label{rem_kkun}
A similar result to Proposition \ref{prop_conv_r2} can also be proved for the Kuramoto model with noise \eqref{kku_n}. The strategy of the proof is analogous but it uses stochastic processes techniques to write the corresponding stochastic differential equation systems. Moreover, one needs to resort to Wasserstein distances instead of the bounded Lipschitz distance above to make the stability argument of the solutions. We do not include the details of the proof since the technicalities lie outside of the scope of the present work, see \cite{BCC} for related problems. 
\end{remark}

Notice that Proposition \ref{prop_conv_r2} and Remark \ref{rem_kkun} allow us to work with continuous frequency distributions for both Kuramoto models \eqref{kku} and \eqref{kku_n} by approximation. More precisely, we can approximate Gaussian and uniform frequency distributions by sums of Dirac Deltas at a finite number of frequencies while approximating the initial data by smooth functions if they are not regular enough. This fact will be used in the numerical schemes in Section 3 and the simulations in Section 4.

\subsection{The Kuramoto model with diffusion: Stationary States and Free Energies}
We first derive an explicit compatibility condition for smooth stationary states $\rho_\infty$ of the equation \eqref{kku_n}. We first easily find from \eqref{kku_n} and \eqref{rew_u} that
\[
D\pa^2_\theta \rho_\infty = \pa_\theta\lt(\lt(\omega - Kr_\infty \sin(\theta - \varphi_\infty)\rt)\rho_\infty\rt), 
\]
where $r_\infty$ and $\varphi_\infty$ are given by
\[
r_\infty e^{i \varphi_\infty} = \int_{\T \times \R} e^{i\theta} \rho_\infty(\theta,\omega)g(\omega)\,d\theta d\omega.
\]
We notice that we can set $\varphi_\infty = 0$ in the stationary state without loss of generality by choosing the right angular reference system. This yields 
\[
\pa_\theta \rho_\infty(\theta,\omega) - \frac{\omega - Kr_\infty \sin \theta}{D} \rho_\infty = c_0(w).
\]
for some function $c_0(w)$. Solving the above differential equation, we find
$$\begin{aligned}
\rho_\infty(\theta,\omega) & = \rho_\infty(0,\omega)\exp\lt(\frac{-Kr_\infty + \omega \theta + Kr_\infty\cos \theta}{D}\rt)\cr
&\qquad \qquad \qquad \qquad \times \lt(1 + \frac{c_0(\omega)e^{\frac{Kr_\infty}{D}}}{\rho_\infty(0,\omega)}\int_0^\theta \exp\lt( -\frac{\omega \theta_* + Kr_\infty \cos\theta_*}{D}\rt)d\theta_*\rt),
\end{aligned}$$
for $\theta \in [0,2\pi)$, where $\rho_\infty(0,\omega)$ is fixed by the normalization, i.e.,  
\[
\int_\T \rho_\infty(\theta,\omega) \,d\theta = 1 \quad \mbox{for all } w \in \R.
\]
On the other hand, since $\rho_\infty(2\pi,\omega) = \rho_\infty(0,\omega)$, we get
\[
1= \exp\lt(\frac{2\pi \omega }{D}\rt)\lt(1 + \frac{c_0(\omega)e^{\frac{Kr_\infty}{D}}}{\rho_\infty(0,\omega)}\int_0^{2\pi} \exp\lt( -\frac{\omega \theta_* + Kr_\infty \cos\theta_*}{D}\rt)d\theta_*\rt),
\]
and subsequently, this implies
\[
\frac{c_0(\omega)e^{\frac{Kr_\infty}{D}}}{\rho_\infty(0,\omega)} = \frac{\exp\lt(-2\pi \omega/D\rt) - 1}{\int_0^{2\pi} \exp\lt( -(\omega \theta_* + Kr_\infty \cos\theta_*)/D\rt)d\theta_*}.
\]
Hence we have
\begin{align}\label{eq_rn}
\begin{aligned}
\rho_\infty(\theta,\omega) = \,&\rho_\infty(0,\omega)\exp\lt(\frac{-Kr_\infty + \omega \theta + Kr_\infty\cos \theta}{D}\rt)\cr
&\times \lt(1 + \frac{\lt(\exp\lt(-2\pi \omega/D\rt) - 1\rt)\int_0^\theta \exp\lt( -(\omega \theta_* + Kr_\infty \cos\theta_*)/D\rt)d\theta_*}{\int_0^{2\pi} \exp\lt( -(\omega \theta_* + Kr_\infty \cos\theta_*)/D\rt)d\theta_*}\rt).
\end{aligned}
\end{align}

Let us discuss further properties in the case of noisy identical Kuramoto oscillators, which are governed by the equation \eqref{kku_n} with $g = \delta_0$. Then we can set $\omega = 0$ without loss of generality in \eqref{eq_rn} to conclude that its stationary state is given by 
\[
\rho_\infty(\theta) = \rho_\infty(0)\exp\lt( \frac{-Kr_\infty + Kr_\infty \cos \theta}{D}\rt),
\]
where $\rho_\infty(0)$ is again fixed by the normalization 
\[
\int_\T \rho_\infty(\theta)\,d\theta = 1.
\]
Note that for both identical and non-identical oscillators if $r_\infty = 0$, i.e., incoherence state, we obtain
\[
\rho_\infty(\theta,\omega) = \rho(0,\omega) = \frac{1}{2\pi}, \quad \mbox{for all } \omega \in \R,
\]
due to the normalization. In this case the equation \eqref{kku_n} has the structure of a gradient flow. More precisely, if we set 
\[
\xi[\rho](\theta,t):= K(W \star \rho)(\theta,t) - D\log \rho(\theta,t),
\]
with $W(\theta)=\cos \theta$, then it is easy to check that
\[
\pa_t \rho + \pa_\theta(\rho\pa_\theta \xi[\rho]) = 0.
\]
Using this observation, we now estimate the following free energy
\[
\mathcal{E}(t) := -\frac K2 \int_\T (W\ast\rho)(\theta,t) \rho(\theta,t)\,d\theta + D\int_\T \rho(\theta,t)\log \rho(\theta,t)\,d\theta.
\]

\begin{lemma}\label{lem_free}Let $\rho$ be a smooth solution to the equation \eqref{kku_n} with $g = \delta_0$. Then we have 
\[
\frac{d}{dt}\mathcal{E}(t) = - \mathcal{D}_{\mathcal{E}}(t),
\]
where the dissipation rate $\mathcal{D}_{\mathcal{E}}(t)$ is given by
\[
\mathcal{D}_{\mathcal{E}}(t) := \int_\T |\pa_\theta \xi[\rho]|^2\rho\,d\theta.
\]
\end{lemma}
\begin{proof}A straightforward computation yields
\[
\frac{d}{dt}\mathcal{E}(t)  = -\int_\T \lt(K(W \star \rho)-D\log\rho  \rt)\pa_t\rho\,d\theta = -\int_\T \xi[\rho] \pa_t \rho\,d\theta = -\int_\T |\pa_\theta \xi[\rho]|^2\rho\,d\theta.
\]
\end{proof}
We next provide the monotonicity of the order parameter for the identical case.
\begin{lemma}\label{lem_free}Let $\rho$ be a smooth solution to the equation \eqref{kku_n} with $g = \delta_0$. Then we have
\[
\dot r = Kr\int_\T \sin^2(\varphi - \theta)\rho d\theta - Dr.
\]
In particular, if the strength of noise $D$ is strong enough such that $D\geq K$, then $\dot r \leq 0$ for all $t \geq 0$.
\end{lemma}
\begin{proof}By definition of the order parameter $r$, we get
\[
\dot r = \dot \varphi \int_\T \sin(\theta - \varphi) \rho(\theta,t)\,d\theta + \int_\T \cos(\theta - \varphi) \pa_t \rho(\theta,t)\,d\theta = I_1 + I_2,
\]
where $I_1$ vanishes since
\[
I_1 = \frac{\dot \varphi}{r}\int_{\T \times \T} \sin(\theta - \theta_*)\rho(\theta,t)\rho(\theta_*,t)\,d\theta d\theta_* = 0.
\]
For the estimate of $I_2$, we find
$$\begin{aligned}
I_2 &= \int_\T \cos(\theta - \varphi) \lt( D\pa_\theta^2\rho - \pa_\theta(u[\rho]\rho) \rt)d\theta 
= - D\int_\T \cos(\theta-\varphi) \rho\,d\theta - \int_\T \sin(\theta - \varphi) u[\rho] \rho\,d\theta\cr
&=-Dr + K\int_\T \sin^2(\theta-\varphi)\rho\,d\theta.
\end{aligned}$$
Combining the above two estimates concludes the desired result.
\end{proof}

Before passing to the construction of numerical methods the following remark should be made.
\begin{remark}
Several works for the continuum Kuramoto model are based on the $g$-weighted kinetic density $f(\theta,\omega,t) := \rho(\theta,\omega,t)g(\omega)$ (see \cite{Caglioti1} for example). In this way, one can rewrite the Kuramoto model \eqref{kku_n} as
\begin{align}
&\partial_t f + \partial_{\theta} (v[f] f) = D\pa^2_\theta f, \quad \mbox{with} \quad f(\theta,\omega,t) = \rho(\theta,\omega,t)g(\omega),
 \label{fff}\\
&v[f](\theta,\omega,t) = \omega - K \int_{\T \times \R} \sin(\theta-\theta_*) f(\theta_*,\omega,t)\,d\theta_* d\omega.\nonumber
\end{align}
Note that the distribution of the natural frequencies is now given by
\[
g(\omega)= \int_\T f(\theta,\omega,t)\,d\theta.
\]
Even if, in the continuation, we will use the form \eqref{kku_n} for the construction of our structure preserving numerical methods, they can be easily reformulated to the form \eqref{fff}. In Section \ref{numex} we will show both $\rho(\theta,\omega,t)$ and $f(\theta,\omega,t)$ in some numerical tests. 
\end{remark}

%
%
%
%
\section{Structure preserving methods}\label{sec_cca}

The goal now is to propose finite volume numerical schemes preserving the structure of gradient flow to the case of identical oscillators and that are generalizable for oscillators with natural frequencies given by a distribution function $g(w)$. To start with, from \eqref{rew_u} and \eqref{kku_n}, it follows that the density $\rho$ satisfies the following continuity equation
\bq\label{ku_cca}
\pa_t \rho = \pa_\theta F[\rho],
\eq
with
\bq
F[\rho] = D\pa_\theta \rho - u\rho,\quad u = \omega + Kr\sin(\varphi - \theta).
\label{eq:flux}
\eq

\subsection{Semi-discrete structure preserving schemes}
Inspired by \cite{BCDS,BD,ChCo,PZ,SG}, we construct a discrete numerical scheme in the variable $\theta$ for the above equation as follows. For $i=1,\cdots,N$, we first introduce a uniform spatial grid $\theta_i \in \T$ such that $\theta_{i+1} - \theta_{i} = \Delta \theta$ and $\theta_{N+k} = \theta_k$ for $k \in \R$. Without loss of generality, we set $\theta_{1/2} = \theta_{N + 1/2} = 0 \equiv 2\pi$, and we then define 
\[
\rho_i(\omega,t) := \frac{1}{\Delta \theta}\int_{\theta_{i-1/2}}^{\theta_{i + 1/2}} \rho(\theta,\omega,t)\,d\theta.
\]
We consider the following approximations for \eqref{ku_cca}
\bq\label{ku_app}
\frac{d}{dt}\rho_i(\omega,t) = \frac{F_{i+1/2}[\rho](\omega,t) - F_{i-1/2}[\rho](\omega,t)}{\Delta \theta} \quad \mbox{for} \quad i=1,\cdots,N,
\eq
where the numerical flux function $F_{i\pm1/2}[\rho](\omega,t)$ is given by
\bq
F_{i + 1/2}[\rho](\omega,t) := D\frac{\rho_{i+1} - \rho_i}{\Delta \theta} - u_{i+1/2}\tilde\rho_{i+1/2},
\label{eq:CC_flux}
\eq
with
\[
u_{i+1/2}(\omega,t) := \frac{1}{\Delta \theta}\int_{\theta_i}^{\theta_{i+1}} u(\theta,\omega,t)\,d\theta, \quad \tilde\rho_{i+1/2} := (1 - \delta_{i+1/2})\rho_{i+1} + \delta_{i+1/2}\rho_i,
\]
and 
\bq\label{eq_app_vel}
u(\theta,\omega,t) = \omega + K\Delta\theta\sum_{j=1}^N\sin(\theta_j - \theta)\int_\R \rho_j(\omega_*,t)g(\omega_*)\,d\omega_*.
\eq
As in \cite{ChCo,PZ} a suitable choice of the weight functions $\delta_{i+1/2}$ yields a method that maintains nonnegativity of the solution (without restrictions on $\Delta \theta$) and preserves the steady state of the system with arbitrary order of accuracy. We will refer to the schemes obtained in this way as Chang-Cooper type schemes \cite{ChCo}.

First, observe that when the numerical flux \eqref{eq:CC_flux} vanishes we get 
\bq
\dfrac{\rho_{i+1}}{\rho_i} = \dfrac{\dfrac{D}{\Delta \theta}+ \delta_{i+1/2}u_{i+1/2}}{\dfrac{D}{\Delta \theta}-(1-\delta_{i+1/2})u_{i+1/2}}.
\label{eq:num_SS}
\eq
Similarly, if we consider the exact flux \eqref{eq:flux}, by imposing ${F}[\rho] \equiv 0$, we have
\[
D\partial_\theta \rho = u\rho.
\]
Integrating the above equation on the cell $[\theta_i,\theta_{i+1}]$ we get
\[
\int_{\theta_i}^{\theta_{i+1}}\dfrac{1}{\rho}\partial_\theta \rho\theta = \int_{\theta_i}^{\theta_{i+1}}\dfrac{u}{D}d\theta,
\]
which gives
\bq\label{eq:quasi_SS}
\dfrac{\rho_{i+1}}{\rho_i} = \exp \left\{ \dfrac{\Delta\theta}{D}u_{i+1/2}\right\}.
\eq
Therefore, by equating \eqref{eq:num_SS} and \eqref{eq:quasi_SS} we recover
\bq
\delta_{i+1/2} = \frac{1}{\xi_{i+1/2}} + \frac{1}{1 - \exp(\xi_{i+1/2})} \quad \mbox{with} \quad \xi_{i+1/2} = -\frac{\Delta \theta}{D} u_{i+1/2}.
\label{eq:delta}
\eq
We can state the following. 
\begin{proposition}
The numerical flux function \eqref{eq:CC_flux} with $\delta_{i+1/2}$ defined by \eqref{eq:delta} vanishes when the corresponding flux (\ref{eq:flux}) is equal to zero over the cell $[\theta_i,\theta_{i+1}]$. Moreover, the nonlinear weight functions $\delta_{i+1/2}$ defined by \eqref{eq:delta} are such that $0 < \delta_{i+1/2} < 1$.  
\end{proposition}
\begin{proof}
The latter result follows from the simple inequality $\exp(x) \geq 1 + x$. 
\end{proof}

\begin{remark}\label{rmk_u}Since
\[
r\cos(\varphi - \theta_i) = \Delta \theta \sum_{j=1}^N \cos(\theta_j - \theta_i)\int_\R \rho_j(\omega,t)g(\omega)\,d\omega,
\]
we get
$$\begin{aligned}
u_{i+1/2}(\omega,t) &= \omega + K\sum_{j=1}^N\lt(\cos(\theta_j - \theta_{i+1}) - \cos(\theta_j - \theta_i) \rt)\int_\R \rho_j(\omega,t)g(\omega)\,d\omega\cr
&=\omega + \frac{Kr}{\Delta \theta}\lt(\cos(\varphi - \theta_{i+1}) - \cos(\varphi - \theta_i)\rt).
\end{aligned}$$
On the other hand, the above expression for $u_{i+1/2}$ is not that useful in practice since it contains $r$ which depends on $\rho$. Thus, it would be more technically useful to write
\bq
\begin{aligned}
u_{i+1/2}(\omega,t) &= \omega + 2K\sin(\Delta \theta/2)\sum_{j=1}^N\sin(\theta_j - \theta_{i+1/2})\int_\R \rho_j(\omega,t)g(\omega)\,d\omega\cr
&=\omega + 2K\sin(\Delta \theta/2)\lt( \cos(\theta_{i+1/2})\sum_{j=1}^N\rho^s_j(t) - \sin(\theta_{i+1/2})\sum_{j=1}^N\rho^c_j(t)\rt),
\label{eq:u}
\end{aligned}
\eq
where we set
\[
\rho^s_j(t) := \sin (\theta_j)\int_\R \rho_j(\omega,t)g(\omega)\,d\omega \quad \mbox{and} \quad \rho^c_j(t) := \cos(\theta_j) \int_\R \rho_j(\omega,t)g(\omega)\,d\omega.
\]

\end{remark}

\begin{remark}
The resulting scheme is second order accurate in $\Delta \theta$ for $D>0$, and degenerate to simple first order upwinding in the limit case $D=0$. In fact, it is immediate to show that as $D\to 0$ we obtain the weights
\[
\delta_{i+1/2}=
\left\{
\begin{array}{cc}
 0, & u_{i+1/2}<0,  \\
 1, & u_{i+1/2}>0.  
\end{array}
\right.
\]
\end{remark}

In the lemma below, we show that the numerical scheme conserves the mass.
\begin{lemma}Consider the numerical scheme \eqref{ku_app}. Then we have
\[
\frac{d}{dt}\sum_{i=1}^N \rho_i(\omega,t) = 0 \quad \mbox{for } (\omega,t)\in\R \times \R_+.
\]
\end{lemma}
\begin{proof}
It follows from the periodicity of domain that
\bq\label{rho1}
\rho_{N+1} = \rho_1 \quad \mbox{and} \quad \rho_N = \rho_0.
\eq
Similarly, we get $u_{N+1/2} = u_{1/2}$ and subsequently this implies $\xi_{N+1/2} = \xi_{1/2}$, $\delta_{N+1/2} = \delta_{1/2}$, and $\tilde\rho_{N+1/2} = \tilde\rho_{1/2}$.
From the above properties, we can easily obtain
\[
F_{N+1/2}[\rho](\omega,t) = F_{1/2}[\rho](\omega,t) \quad \mbox{for } (\omega,t)\in\R \times \R_+.
\]
This, together with the following straightforward computation 
$$\begin{aligned}
\frac{d}{dt}\sum_{i=1}^N \rho_i(\omega,t) &= \frac{1}{\Delta \theta}\sum_{i=1}^N\lt(F_{i+1/2}[\rho](\omega,t) - F_{i-1/2}[\rho](\omega,t)\rt) \cr
&= \frac{1}{\Delta \theta} \lt(F_{N+1/2}[\rho](\omega,t) - F_{1/2}[\rho](\omega,t) \rt) = 0,
\end{aligned}$$
gives the desired result.
\end{proof}
We next provide the positivity preservation whose proof can be obtained by using almost same argument as in \cite[Proposition 1]{PZ}. However, for the completeness of this work, we sketch the proof in the proposition below. For this, we introduce the time discretization $t^n = n\Delta t$ with $\Delta t >0$ and $n\in \N_0$ and consider the following forward Euler method
\bq\label{ku_app2}
\rho_i^{n+1}(\omega,t) = \rho_i^n(\omega,t) + \Delta t\frac{F_{i+1/2}^n[\rho](\omega,t) - F_{i-1/2}^n[\rho](\omega,t)}{\Delta \theta} \quad \mbox{for} \quad i=1,\cdots,N,
\eq
where
\[
F_{i + 1/2}^n[\rho](\omega,t) := D\frac{\rho^n_{i+1} - \rho^n_i}{\Delta \theta}- u_{i+1/2}^n\tilde\rho_{i+1/2}^n.
\]
\begin{proposition}\label{prop_pos} Suppose that $g$ is compactly supported and the time step $\Delta t$ satisfies 
\[
\Delta t \leq \frac{(\Delta \theta)^2}{2(C_0\Delta \theta + D)} \quad \mbox{where} \quad C_0 = |supp(g)| + K.
\]
Then the explicit scheme \eqref{ku_app2} preserves nonnegativity, i.e., $\rho_i^{n+1} \geq 0$ if $\rho_i^n \geq 0$ for $i=1,\cdots,N$.
\end{proposition}
\begin{proof} It follows from \eqref{ku_app2} that
$$\begin{aligned}
&F_{i+1/2}^n[\rho](\omega,t) - F_{i-1/2}^n[\rho](\omega,t)\cr
&\quad = \rho_{i+1}^n\lt(\frac{D}{\Delta \theta}- u_{i+1/2}^n(1 - \delta_{i+1/2}^n) \rt) + \rho_i^n\lt( u_{i-1/2}^n(1 - \delta_{i-1/2}^n) -u_{i+1/2}^n \delta_{i+1/2}^n -\frac{2D}{\Delta \theta} \rt) \cr
&\qquad + \rho_{i-1}^n\lt(\frac{D}{\Delta \theta} + u_{i-1/2}^n\delta_{i-1/2}^n \rt).
\end{aligned}$$
On the other hand, we easily find
\[
\frac{D}{\Delta \theta}- u_{i+1/2}^n(1 - \delta_{i+1/2}^n) = \frac{D \xi_{i+1/2}}{\Delta \theta}\lt(1 - \frac{1}{1 - \exp\lt(\xi_{i+1/2}\rt)} \rt) \geq 0,
\]
due to $x(1 - (1-e^x)^{-1}) \geq 0$ for $x \in \R$. Similarly, we get
\[
\frac{D}{\Delta \theta} + u_{i-1/2}^n\delta_{i-1/2}^n = \frac{D}{\Delta \theta}\lt(\frac{\xi_{i-1/2}}{\exp\lt(\xi_{i-1/2} \rt)-1} \rt) \geq 0.
\]
We also notice that 
$$\begin{aligned}
\max_{0 \leq i \leq N}\lt| u_{i-1/2}^n(1 - \delta_{i-1/2}^n)-u_{i+1/2}^n \delta_{i+1/2}^n  -\frac{2D}{\Delta \theta}\rt| &\leq \max_{0 \leq i \leq N}\lt(|u_{i+1/2}^n| + |u_{i-1/2}^n| \rt) + \frac{2D}{\Delta \theta}\cr
&\leq 2\lt(|supp(g)| + K + \frac{D}{\Delta \theta}\rt).
\end{aligned}$$
This, together with the property of convex combination, yields that the nonnegativity is preserved if the time step $\Delta t$ satisfies 
\[
\Delta t \leq \frac{\Delta \theta}{2\lt(|supp(g)| + K + D/\Delta \theta \rt)}.
\]
\end{proof}
\begin{remark}\label{rmk_si} The parabolic stability restriction $\Delta t = \mathcal{O}(\Delta \theta^2)$ which appears in Proposition \ref{prop_pos} can be avoided if we use a semi-implicit scheme as in \cite{PZ}. To be more precise, let us consider 
\bq\label{ku_app3}
\rho_i^{n+1}(\omega,t) = \rho_i^n(\omega,t) + \Delta t\frac{\hat F_{i+1/2}^{n+1}[\rho](\omega,t) - \hat F_{i-1/2}^{n+1}[\rho](\omega,t)}{\Delta \theta} \quad \mbox{for} \quad i=1,\cdots,N,
\eq
where
\[
\hat F_{i + 1/2}^{n+1}[\rho](\omega,t) := D\frac{\rho^{n+1}_{i+1} - \rho^{n+1}_i}{\Delta \theta} - u_{i+1/2}^n\lt((1 - \delta^n_{i + 1/2})\rho^{n+1}_{i+1} + \delta^n_{i+1/2}\rho^{n+1}_i\rt).
\]
It follows from \eqref{ku_app3} that
$$\begin{aligned}
\rho_i^n &=-\frac{\Delta t}{\Delta \theta}\lt(u_{i-1/2}^n \delta_{i-1/2}^n + \frac{D}{\Delta \theta} \rt)\rho_{i-1}^{n+1} \cr
&\quad + \lt(1 - \frac{\Delta t}{\Delta \theta}\lt(u_{i-1/2}^n(1 - \delta_{i-1/2}^n)-u_{i+1/2}^n \delta_{i+1/2}^n - \frac{2D}{\Delta \theta} \rt) \rt)\rho_i^{n+1}\cr
&\quad + \frac{\Delta t}{\Delta \theta}\lt(\frac{D}{\Delta \theta} -\lt(1 - \delta_{i+1/2}^n \rt)u_{i+1/2}^n \rt)\rho_{i+1}^{n+1}\cr
&=:A_i^n \rho_{i-1}^{n+1} + B_i^n \rho_i^{n+1} + C_i^n\rho_{i+1}^{n+1},
\end{aligned}$$
for $i=1, \cdots, N$. We then define a matrix $M^n \in \R^N \times \R^N$ by
\begin{equation}
M^n = \left( \begin{array}{ccccccc} B_1^n & C_1^n & 0 & \ldots & \ldots & 0 & A_1^n \\ 
A_2^n & B_2^n & C_2^n & 0 & \ldots & \ldots & 0 \\
0 & A_3^n & B_3^n & C_3^n & \ddots &  & \vdots \\
\vdots  & \ddots & \ddots & \ddots & \ddots & \ddots & \vdots \\
\vdots  &  & \ddots & \ddots & \ddots & \ddots & 0 \\
0  & \ldots & \ldots & 0 & A_{N-1}^n & B_{N-1}^n & C_{N-1}^n \\
C_N^n & 0 & \ldots & \ldots & 0  & A_N^n & B_N^n \\
\end{array}\right).
\label{eq:cyclic}
\end{equation}
If we denote $\bar\rho^n = (\rho_1^n, \rho_2^n, \cdots, \rho_N^n) \in \R^N $, we obtain 
\[
M^n \bar\rho^{n+1} = \bar\rho^n, \quad \mbox{i.e.,} \quad \sum_{j=1}^NM^n_{ij}\rho_j^{n+1} = \rho_i^n \quad \mbox{for} \quad i =1,\cdots,N.
\]
Thus, the matrix $M^n$ is invertible and we have
\[
\bar\rho^{n+1} = (M^n)^{-1}\bar\rho^n.
\]
Then, by the same arguments as in \cite{PZ}, if the time step $\Delta t$ satisfies 
\bq\label{rest_si}
\Delta t < \frac{\Delta \theta}{2C_0},
\eq
where $C_0 >0$ appeared in Proposition \ref{prop_pos}, then the semi-implicit scheme \eqref{ku_app3} preserves the nonnegativity. It is also clear that that the scheme \eqref{ku_app3} conserves the mass. {Finally, concerning the computational cost, it is worth to mention that the usual Thomas algorithm for the solution of a tridiagonal system in $O(N)$ operations can be extended to cyclic tridiagonal matrices of the form \eqref{eq:cyclic} thanks to the Sherman-Morrison formula. We refer to \cite{NR}, Section 2.7.2, for more details.} 
\end{remark}
\subsection{Discrete free energy dissipation}
Next, we present the discrete free energy estimate for the case of identical oscillators. Set
\[
\mathcal{E}_\Delta(t) := -\frac{\Delta \theta}{2}\sum_{i=1}^N \lt( K\Delta \theta\sum_{j=1}^N\cos(\theta_j - \theta_i)\rho_i\rho_j - D\rho_i\log \rho_i\rt).
\]
Taking the time derivative to $\mathcal{E}_\Delta(t)$ gives
\[
\frac{d}{dt}\mathcal{E}_\Delta(t) = -\Delta \theta\sum_{i=1}^N \xi_i[\rho]\frac{d \rho_i}{dt}:= -\Delta \theta\sum_{i=1}^N\lt(K\Delta \theta\sum_{j=1}^N\cos(\theta_j - \theta_i)\rho_j - D\log \rho_i \rt)\frac{d \rho_i}{dt}.
\]
Note that $\xi_1[\rho] = \xi_{N+1}[\rho]$ due to \eqref{rho1}. This together with \eqref{ku_app} yields
\begin{align}\label{fe1}
\frac{d}{dt}\mathcal{E}_\Delta(t) &= - \sum_{i=1}^N \xi_i[\rho](t)\lt(F_{i+1/2}[\rho](t) - F_{i-1/2}[\rho](t) \rt)\nonumber\\
&= -\sum_{i=1}^N\lt(\xi_i[\rho](t) - \xi_{i+1}[\rho](t) \rt)F_{i+1/2}[\rho](t).
\end{align}
On the other hand, it follows from Remark \ref{rmk_u} that
\[
u_{i+1/2} = \frac{1}{\Delta \theta}\lt(\xi_{i+1}[\rho](t) - \xi_i[\rho](t) \rt) + \frac{D}{\Delta \theta}\log\frac{\rho_{i+1}}{\rho_i}.
\]
Thus we get
\bq\label{fe2}
F_{i+1/2}[\rho] = \frac{D}{\Delta \theta}\lt( \rho_{i+1} - \rho_i - \tilde\rho_{i+1/2}\log\frac{\rho_{i+1}}{\rho_i} \rt) - \frac{1}{\Delta \theta}\lt(\xi_{i+1}[\rho](t) - \xi_i[\rho](t) \rt) \tilde\rho_{i+1/2}.
\eq
We then combine \eqref{fe1} and \eqref{fe2} to find
$$\begin{aligned}
\frac{d}{dt}\mathcal{E}_\Delta(t) &= -\frac{1}{\Delta\theta}\sum_{i=1}^N\lt(\xi_i[\rho](t) - \xi_{i+1}[\rho](t) \rt)^2\tilde\rho_{i+1/2}(t) \cr
&\quad - \frac{D}{\Delta\theta} \sum_{i=1}^N\lt(\xi_i[\rho](t) - \xi_{i+1}[\rho](t)\rt)\lt(\rho_{i+1}(t) - \rho_i(t) -\tilde\rho_{i+1/2}(t)\log\frac{\rho_{i+1}}{\rho_i}\rt).
\end{aligned}$$
We notice that if $\tilde\rho_{i+1/2}$ is given through the following relation 
\bq\label{fe3}
\rho_{i+1}(t) - \rho_i(t) -\tilde\rho_{i+1/2}(t)\log\frac{\rho_{i+1}}{\rho_i} = 0,
\eq
we have the discrete dissipation of free energy estimate, however, this is not the case for 
\[
\tilde\rho_{i+1/2} = (1 - \delta_{i+1/2})\rho_{i+1} + \delta_{i+1/2}\rho_i,
\]
with $\delta_{i+1/2}$ defined by \eqref{eq:delta}. On the other hand, if we choose \cite{BD,PZ}
\bq\label{flux1}
\tilde\rho^E_{i+1/2} = (1 - \delta^E_{i+1/2})\rho_{i+1} + \delta^E_{i+1/2}\rho_i \quad \mbox{with} \quad \delta^E_{i+1/2} = \frac{\rho_{i+1}}{\rho_{i+1} - \rho_i} + \frac{1}{\log\rho_i - \log \rho_{i+1}},
\eq
then $\tilde\rho^E_{i+1/2}$ satisfies the relation \eqref{fe3}, and subsequently this yields
\[
\frac{d}{dt}\mathcal{E}_\Delta(t) = -\mathcal{D}_\Delta(t),
\]
where the discrete dissipation rate $\mathcal{D}_\Delta(t)$ is given by
\bq
\mathcal{D}_\Delta(t) := \frac{1}{\Delta\theta}\sum_{i=1}^N\lt(\xi_i[\rho](t) - \xi_{i+1}[\rho](t) \rt)^2\tilde\rho_{i+1/2}^E(t).
\label{eq:diss}
\eq

\begin{proposition} Let us consider in \eqref{ku_app} the numerical flux function 
\bq
F_{i + 1/2}^E[\rho](\omega,t) := \left(\frac{D}{\Delta \theta}\log\frac{\rho_{i+1}}{\rho_i}- u_{i+1/2}\right)\tilde\rho^E_{i+1/2},
\label{eq:E_flux}
\eq
with $\tilde\rho^E_{i+1/2}$ defined in \eqref{flux1}. Then, we have the following semidiscrete free energy estimate
$$
\frac{d}{dt}\mathcal{E}_\Delta(t) = -\mathcal{D}_\Delta(t), \quad t\geq 0,
$$
with $\mathcal{D}_\Delta(t)$ given by \eqref{eq:diss}.
\end{proposition}

\begin{remark}
The weights defined by \eqref{flux1} are such that $0 < \delta^E_{i+1/2} < 1$, moreover, it is easy to verify that the numerical flux function \eqref{eq:E_flux} vanishes when the corresponding flux \eqref{eq:flux} is equal to zero over the cell $[\theta_i,\theta_{i+1}]$. On the other hand, nonnegativity of the solution, is satisfied only under more restrictive conditions. In fact, similar to central differences, we have a restriction on the mesh size $\Delta \theta$ which becomes prohibitive for small values of the diffusion function $D$. We refer to \cite{PZ} for further details.
\end{remark}

\subsection{Frequency discrete schemes}
In order to derive fully discrete schemes we must discuss the discretization of the
frequency variable $\omega \in \R$. Since, the computation of the order parameter and the flux function are obtained as integrals in $\omega$ with respect to the probability density $g(\omega)$ it is natural to consider Gaussian quadrature formulas based on orthogonal polynomials. This technique, which corresponds to a collocation method, was previously used for linear transport equation and semi-conductors models and is closely related to moment methods \cite{DP15,JP1,LM,RSZ}. Moreover, from the perspective of Proposition \ref{prop_conv_r2}, we are approximating $g$ by a suitable weighted average of Dirac Deltas at chosen frequencies as specified below.

Here, for reader's convenience, we recall some basic facts concerning orthogonal polynomials and Gaussian quadrature. Let us consider an orthogonal system of polynomials $\big\{H_n(\omega)\big\}_{n\in\N}$,
with respect to the probability density function $g(\omega)$ 
\[
\int_{\R} H_n(\omega) H_m(\omega) g(\omega) d\omega = c_n \delta_{mn},\quad m,n\in \N,
\]
where $\delta_{mn}$ is the Kronecker delta and $c_n$ are normalization constants 
\[
c_n = \int_{\R} H_n(\omega)^2 g(\omega)d\omega.
\]
Of particular interest for applications to the Kuramoto system are the Legendre polynomials in the case of a uniform distribution of frequencies and the Hermite polynomials for normally distributed frequencies (see for example \cite{X}). Now let $\omega_k$, $k=1,\ldots,M$, be the zeros of $H_M$ and let $l_k(\cdot)$ be the $(M-1)$th degree Lagrange polynomial through the nodes $\omega_k$, $k=1,\ldots,M$. 
Therefore, the integrals $\rho_j^s(t)$ and $\rho_j^c(t)$ in \eqref{eq:u} can be approximated taking
\bq
\int_{\R} \rho_j(\omega,t)g(\omega)\,d\omega \approx \sum_{k=1}^M \rho_j(\omega_k,t)g_k, 
\label{eq:gauss}
\eq
where the weights $g_k$ are given by
\[
g_k = \int_{\R} l_k(\omega) g(\omega)\,d\omega.
\]
It is well-known that \eqref{eq:gauss} becomes an equality if $\rho_j(\omega,t)$ is any
polynomial of degree less than or equal to $2M-1$ in $\R$. Notice, as mentioned earlier, that for general $\rho_j(\omega,t)$, these formulas are equivalent to say that we approximate our frequency distribution by
$$
g^M (\omega) = \sum_{k=1}^M g_k \delta_{\omega_k}\,.
$$

The resulting numerical scheme for the Kuramoto model is therefore a collocation method of the form
\bq\label{ku_app2}
\frac{d}{dt}\rho_i(\omega_k,t) = \frac{F_{i+1/2}[\rho](\omega_k,t) - F_{i-1/2}[\rho](\omega_k,t)}{\Delta \theta}, \qquad  i=1,\cdots,N,
\eq
where $F_{i+1/2}[\rho](\omega_k,t)$ are the numerical fluxes at the collocation points $\omega_k$, $k=1,\ldots,M$ given by \eqref{eq:CC_flux} for the Chang-Cooper type schemes or by \eqref{eq:E_flux} for the entropic schemes. Clearly, all the main physical properties discussed before, like nonnegativity and preservation of the steady states are maintained by the corresponding collocation scheme.

\section{Numerical examples}
\label{numex}
%
In this section, we present several numerical experiments showing phase transitions of the order parameter $r^\infty$, which is defined by 
\[
r^\infty:= \lim_{t \to \infty} r(t) = \lim_{t \to \infty} \lt|\int_{\T \times \R} e^{i \theta} \rho(\theta,\Omega,t) g(\Omega) \,d\theta d\Omega\rt|,
\]
for the continuum Kuramoto equation \eqref{kku} solved with the structure preserving introduced in Section \ref{sec_cca}. It is worth mentioning that the semi-implicit scheme described in Remark \ref{rmk_si} requires the time step $\Delta t = \mathcal{O}(\Delta \theta)$, see \eqref{rest_si} whereas the parabolic restriction $\Delta t = \mathcal{O}(\Delta \theta^2)$ is needed for the explicit scheme. For that reason, we employ the semi-implicit scheme to investigate the large time behavior of solutions numerically. We refer to this method as Implicit Structure Preserving  (ISP) scheme. We only use the explicit scheme presented in Proposition \ref{prop_pos}, hereafter denoted as ESP scheme, for the time dependent solutions of $\rho$ in Figures \ref{Fig_1}, \ref{Fig_rho_uni}, and \ref{Fig_rho_gau} since it provides second order accuracy in both $t$ and $\vartheta$. Of course, the stationary behavior is computed exactly by both methods. For comparison purposes we compute a direct particle simulation of the noisy Kuramoto system where the diffusion part is solved using a standard random Brownian motion \cite{PT2}. We refer to this method as Particle Monte Carlo (PMC) scheme. {For identical oscillators, to emphasize some of the properties of the new schemes, we compute also the solution with the spectral solver developed in \cite{ABPRS} which essentially consist in a standard Fourier-Galerkin approximation to the continuum Kuramoto equation \eqref{kku}. We refer to this method as Fourier-Galerkin Spectral (FGS) method (see \cite[Section VI. B]{ABPRS} for more details).}
In all the numerical results, thanks to their favorable stability properties, we used the Chang-Cooper type fluxes defined in \eqref{eq:CC_flux} and \eqref{eq:delta}. Analogous results are obtained using the  
fluxes in \eqref{eq:E_flux} and \eqref{flux1}.
In all graphs in Figures \ref{Fig_1}, \ref{Fig_rho_uni}, \ref{Fig_rho_gau} and \ref{Fig_rho_bi}, the black curves are the reference solutions obtained with $N=510$ grid points.

\begin{figure}[ht]
\centering
\includegraphics[scale=0.38]{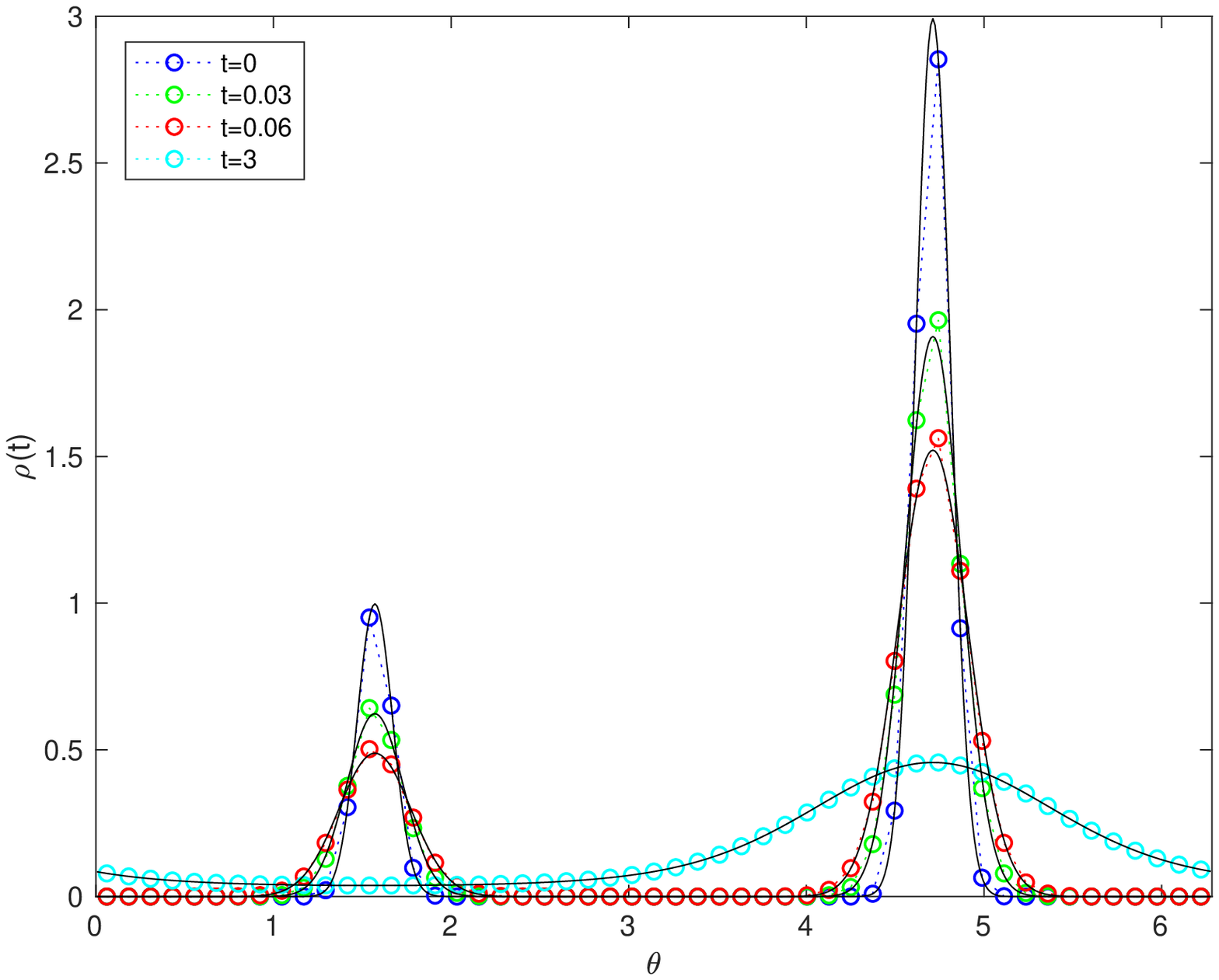}
\includegraphics[scale=0.38]{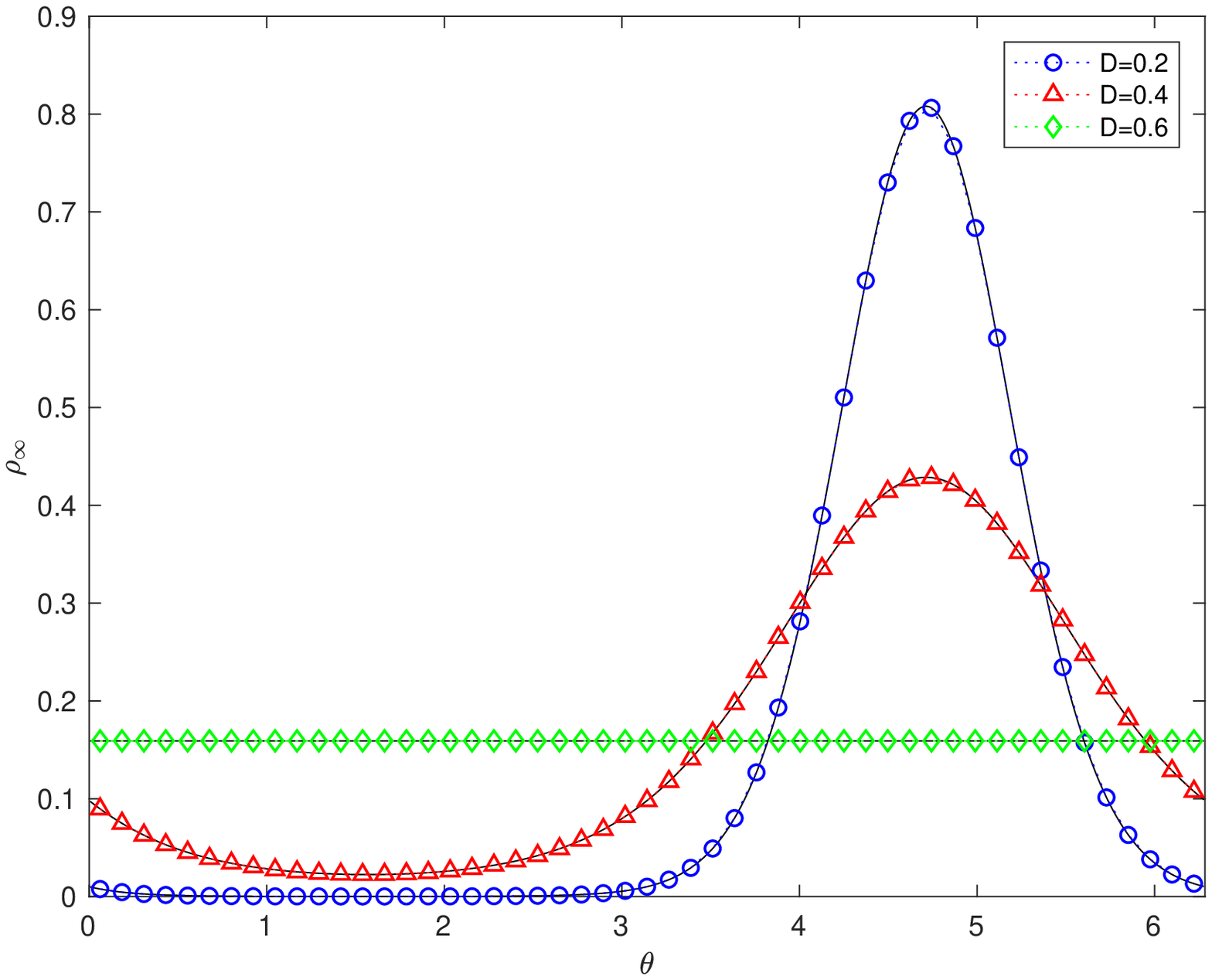}
\caption{Time evolution of solution $\rho(t)$ with $K=1$ and $D=0.25$ with the ESP scheme and $N=51$ (left). Different steady states $\rho_\infty$ with $K=1$ for various values of $D$ with the ISP scheme and $N=51$ (right). (For interpretation of the colors in the figures, the reader is refereed to the web version of this article.)
}
\label{Fig_1}
\end{figure}

\begin{figure}[ht]
\centering
  \pgfmathsetlength{\imagewidth}{\linewidth}%
    \pgfmathsetlength{\imagescale}{\imagewidth/524}%
    \begin{tikzpicture}[x=\imagescale,y=-\imagescale]
    \node[anchor=north west] at (0,0) {\includegraphics[width=0.5\imagewidth]{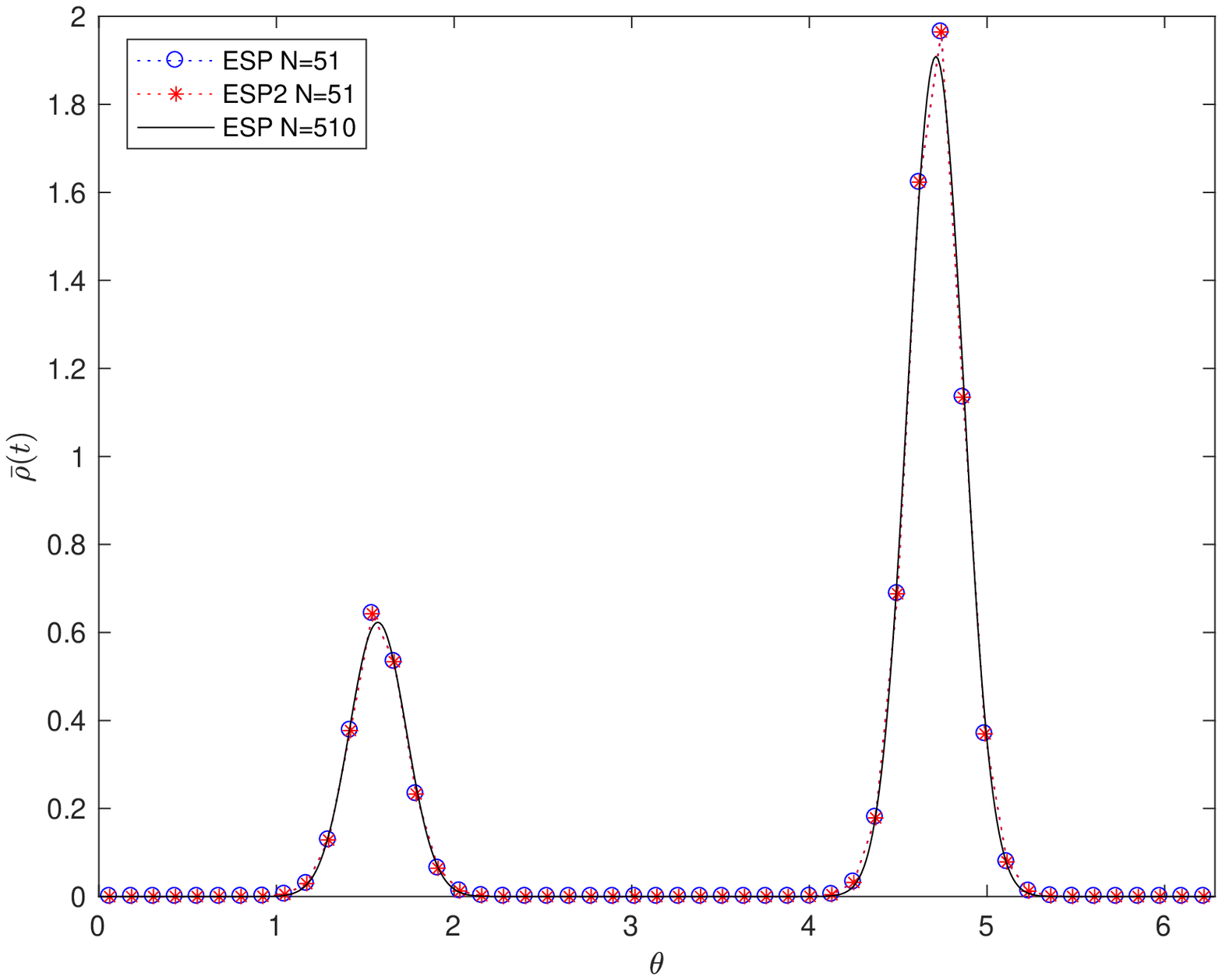}};
        \node[anchor=north west] at (260,0) {\includegraphics[width=0.5\imagewidth]{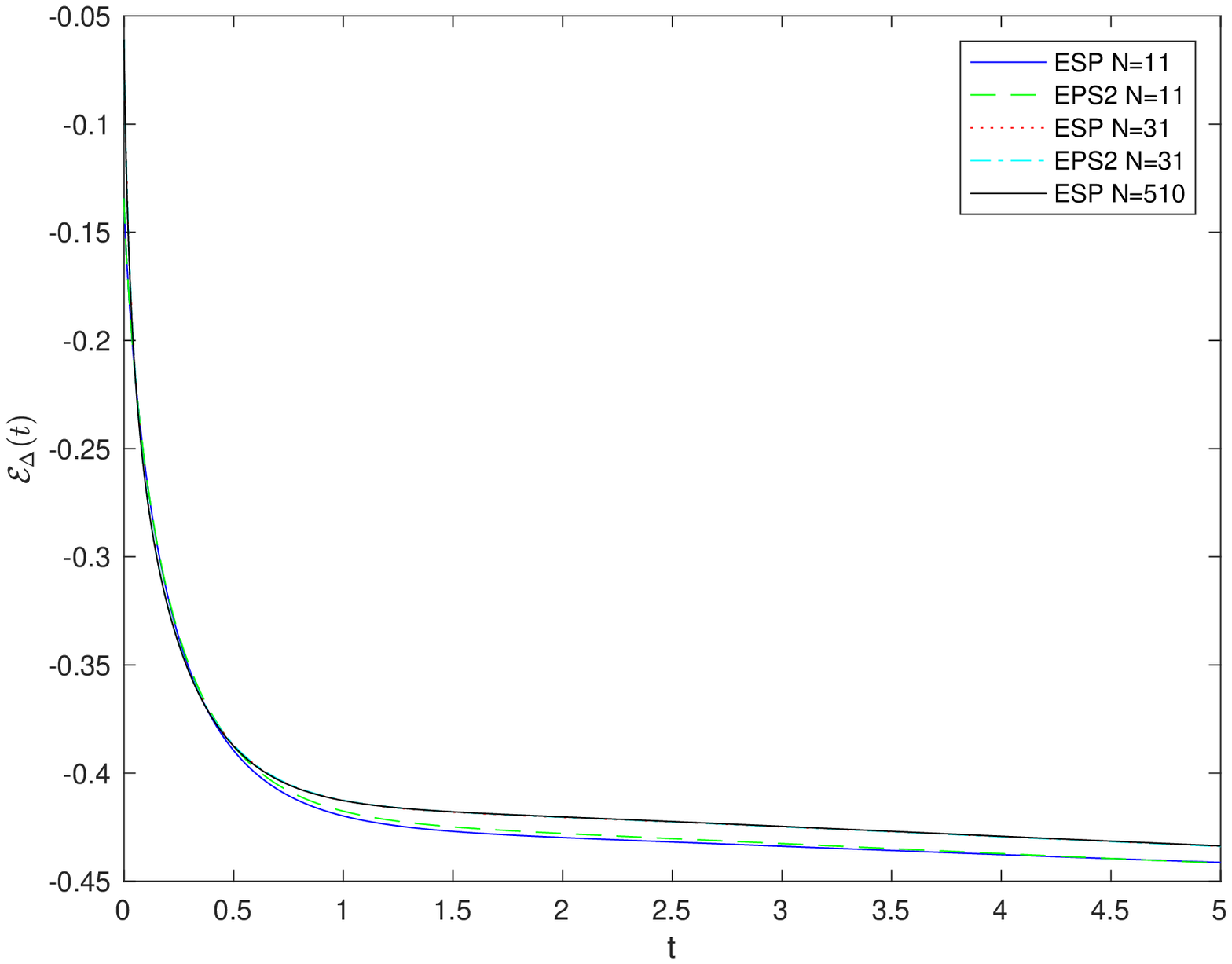}};
        \node[anchor=north west] at (310,40) {\includegraphics[width=0.25\imagewidth]{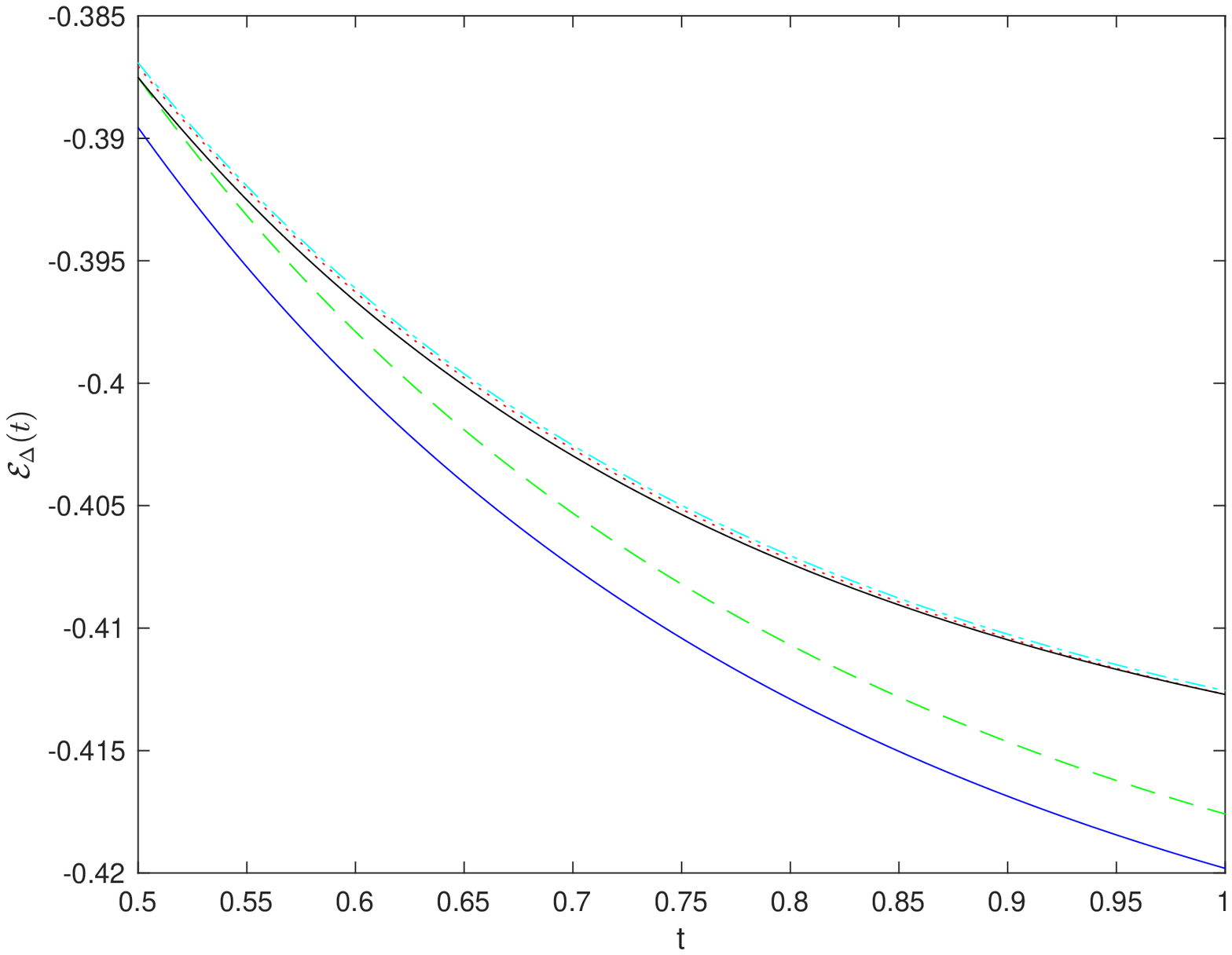}};
    \end{tikzpicture}
%
\caption{
Comparison between fluxes with different $\delta_{i+1/2}$ \eqref{eq:delta} and \eqref{flux1}, labeled ESP2, with $K=1$, $D=0.25$ and $N=51$. Solution at $t = 0.03$ (left) and time evolution of free energies $\mathcal{E}_\Delta(t)$  using different meshes.}
\label{Fig_1_5}
\end{figure}

\begin{figure}[ht]
\centering
\includegraphics[scale=0.4]{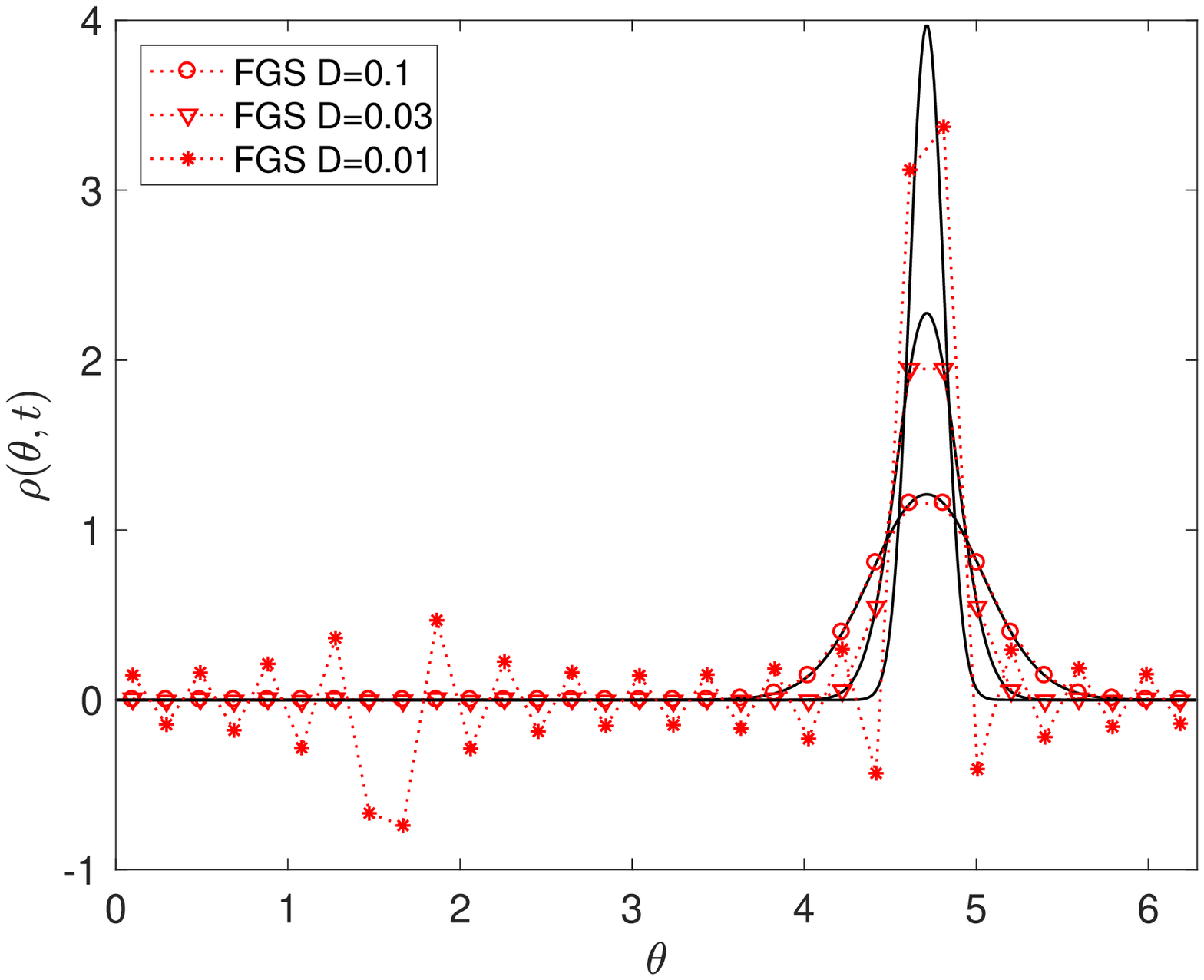}
\hskip 1cm
\includegraphics[scale=0.4]{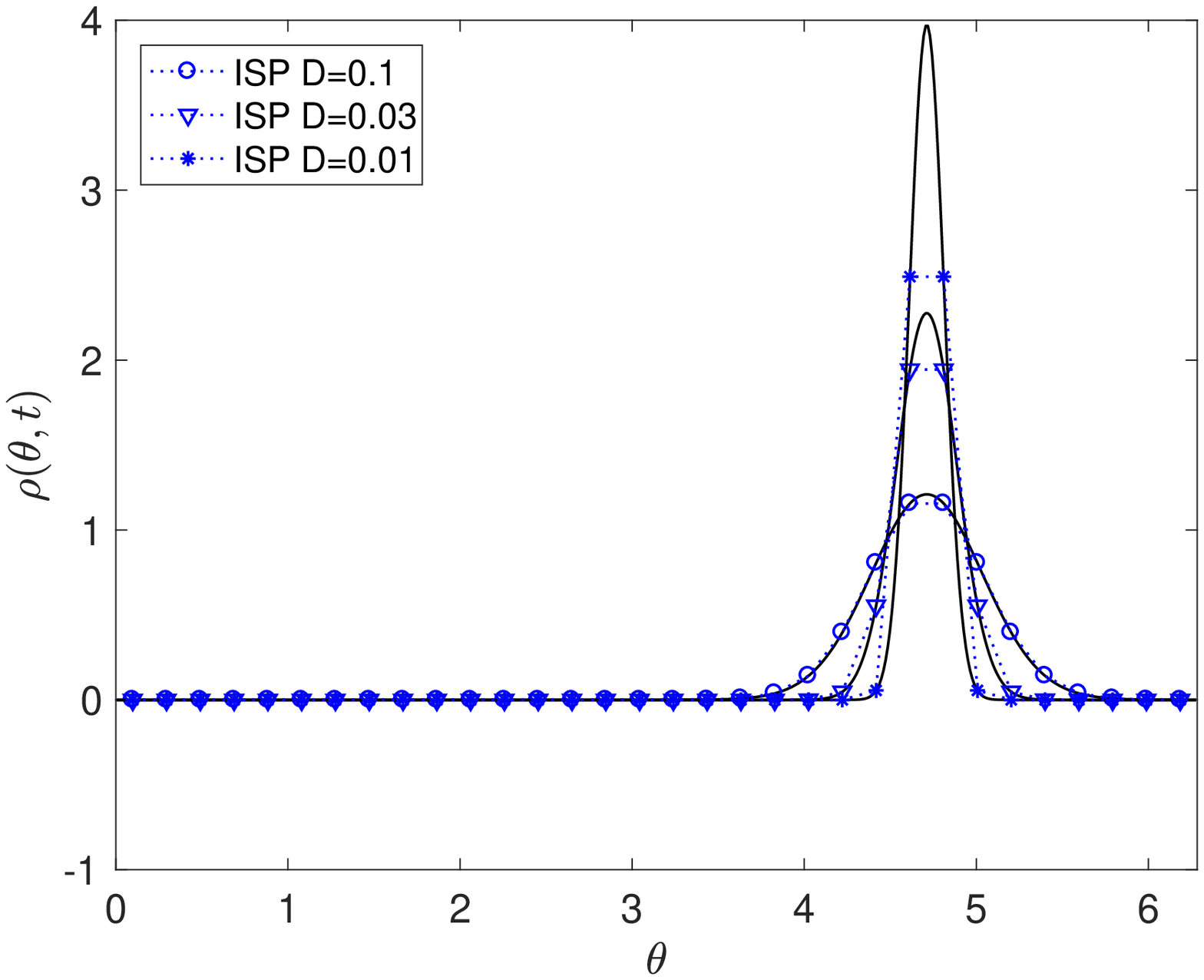}
\caption{{Comparison between the FGS method (left) and the ISP method (right). Steady states for $N=32$, $K=1$ and decreasing values of $D=0.1$, $0.03$, $0.01$.}}
\label{Fig_1b}
\end{figure}

\subsection{Identical Kuramoto oscillators} 
First, to test the validity of our method, we present numerical simulations for the identical Kuramoto oscillators with diffusion. 

It is known that the identical Kuramoto equation \eqref{kku_n} with $g = \delta_0$ exhibits a phase transition at $K = K_c := 2D$, i.e., the oscillators become uniformly distributed which implies $r^\infty \equiv 0$ for $K \leq K_c$, on the other hand, for $K > K_c$ the oscillators converge to a phase-locked state which have a positive order parameter $r^\infty > 0$, see  \cite{ABPRS, GLP, GPP} for more discussion. In order to observe the phase transition, we consider as initial data a symmetric sum of two Gaussians:
\[
\rho_0(\theta) = \frac{c_0}{\sqrt{2\pi\sigma}}\lt(\rho^1 e^{-\frac{1}{2\sigma}\lt(\theta - \frac\pi2 \rt)} + \rho^2 e^{-\frac{1}{2\sigma}\lt(\theta - \frac{3\pi}{2} \rt)} \rt) \quad \mbox{with} \quad \rho^1 = 0.25, \ \rho^2 = 0.75, \ \sigma = 0.01,
\]
where $c_0>0$ is fixed by the normalization, i.e.,
\[
\int_\T \rho_0(\theta)\,d\theta = 1,
\]
and the nodes $\theta_i$, $i=1,\ldots, N$, are chosen equally spaced inside $[0,2\pi]$.

Let us first fix the coupling strength $K = 1$ and vary the strength of the diffusion $D$. Note that, if we introduce a scaled time variable $\tau := K t$, then the identical Kuramoto equation \eqref{kku_n} with $g = \delta_0$ can be rewritten as 
\[
\pa_t \rho - \pa_\theta ((\sin \star \rho)\rho) = \frac DK \pa_\theta^2\rho.
\] 
This yields that if we fix the strength of noise $D$ and varies the coupling strength $K$, then we get the opposite behavior of solutions. 

In Figure \ref{Fig_1}, we show the time evolution of the solution $\rho$ for $K=1$ and $D = 0.25$ with $N=51$ (left) and the steady states solutions $\rho_\infty$ with different values of $D=0.2$, $0.4$, $0.6$ and fixed $N=51$ (right).

 We also compare the time evolution of solutions $\rho(t)$ with the two different numerical fluxes, \eqref{eq:delta} and \eqref{flux1}, presented in Section 3.2 in Figure \ref{Fig_1_5}. As expected the two fluxes are in good agreement and provide essentially the same result. In the same figure, we show the time evolution of free energies with the two different fluxes. To emphasize the differences we used a very rough mesh with $N=11$. They share the nonincreasing property of the free energy and converge to the same value as the mesh is refined.
 
{Next, to emphasize the structure preserving properties of the new schemes, in Figure \ref{Fig_1b}, we report the same kind of computations for decreasing values of the diffusion coefficient $D$ together with the results obtained with the Fourier-Galerkin spectral method. 
Clearly, the FGS method is not capable to compute the solution for small values of $D$ unless the grid size resolves the size of the diffusion coefficient and we can observe the formation of oscillations which produce negative values of the solution. On the contrary the ISP method is robust even for small values of $D$ and provides nonnegative solutions under the stability condition \eqref{rest_si}.
}

{In Table \ref{tab:1} we report the discrete $L_1$ error at the steady state for the ISP scheme and the FGS method for $K=1$ and $D=0.1$. 
Note that $64$ modes are required to the spectral method to match the accuracy of the structure preserving scheme at the steady state. Here the reference solution was computed with the spectral method using $N=4096$ grid points. Clearly, for smaller values of $D$ a larger number of modes is required by the FGS method to have a comparable accuracy with the ISP scheme.}

\begin{table}[htp]
{
\caption{$L_1$ error at the stationary state for $K=1$ and $D=0.1$ }
\begin{center}
\begin{tabular}{cccc}
\hline
 & N=16 & N=32 & N=64\\
\hline\\[-.25cm]
ISP & $7.54\times 10^{-6}$ & $2.61\times 10^{-11}$ & $2.60\times 10^{-11}$ \\
FGS & $4.92\times 10^{-2}$ & $4.53\times 10^{-6}$ & $2.46\times 10^{-11}$ \\[+.1cm]
\hline 
\end{tabular}
\end{center}
\label{tab:1}
}
\end{table}%

\begin{figure}[t]
\centering
\includegraphics[scale=0.38]{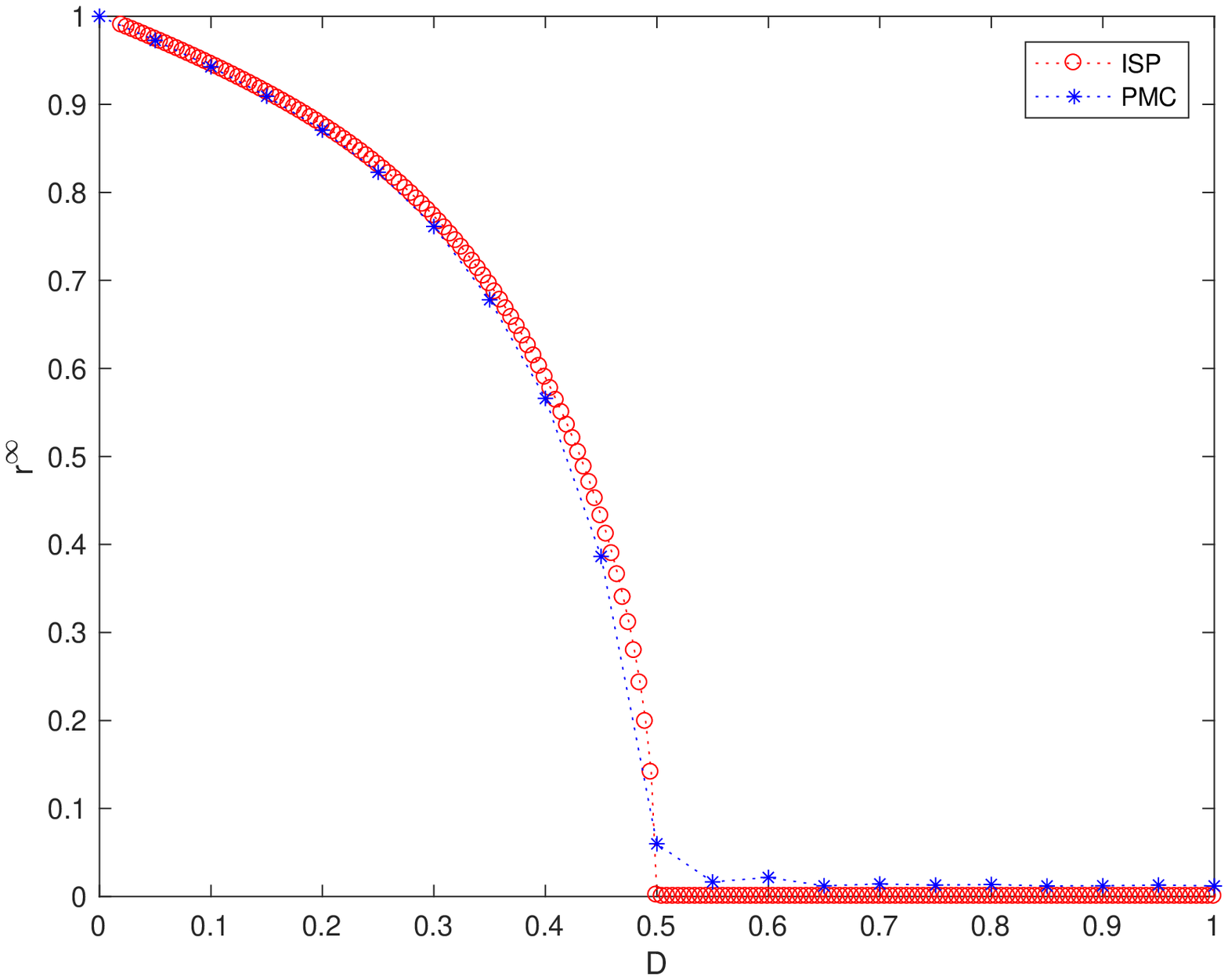}
\includegraphics[scale=0.38]{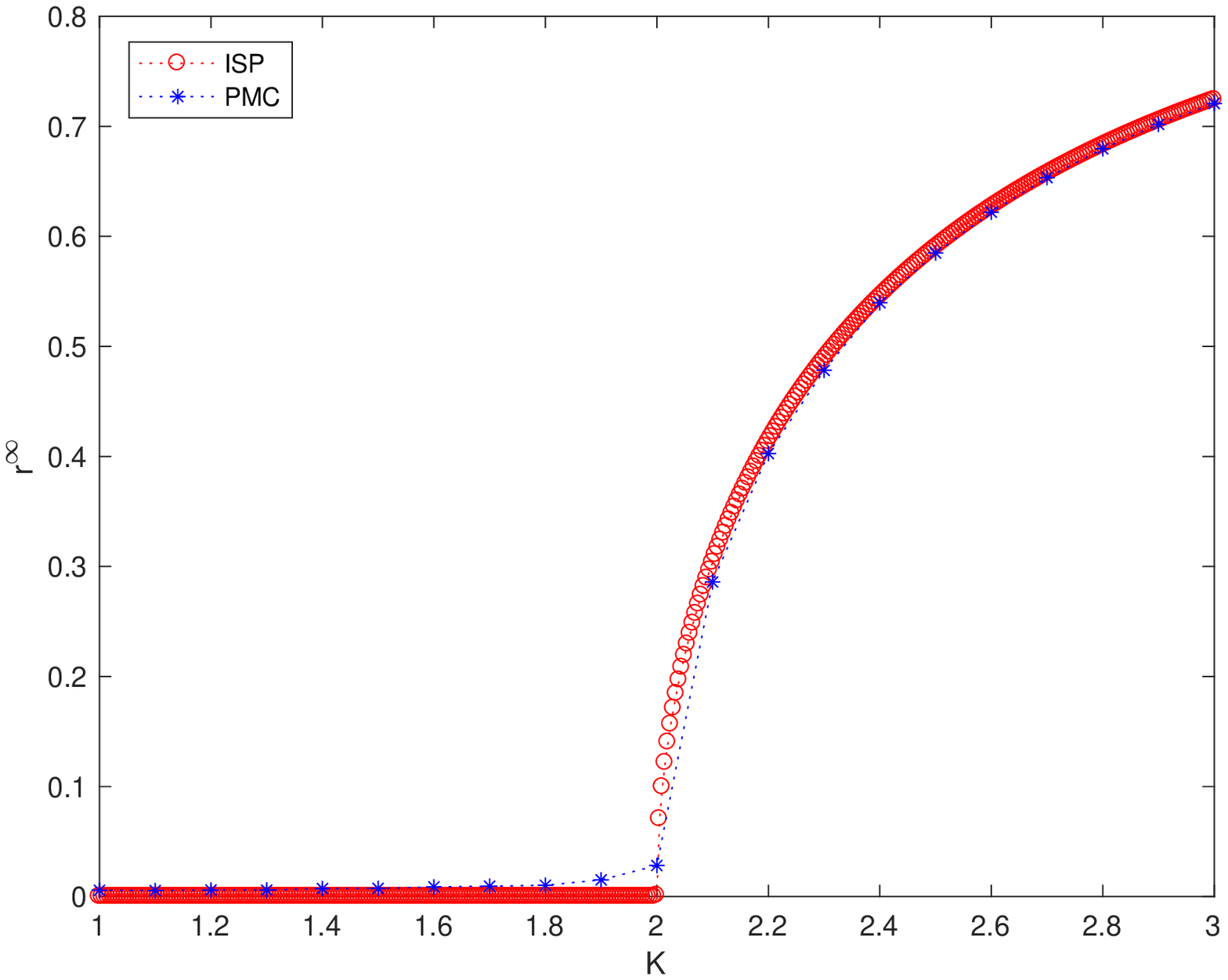}
\caption{Evolution of the order parameter $r^\infty(D)$ for fixed $K=1$ (left). Time evolution of solution $r^\infty(K)$ with $D=1$ (right). The circles refer to the ESP scheme with $N=50$ whereas the asteriscs to the PMC method with $5\times 10^4$ particles and $1000$ averages.} 
\label{Fig_2}
\end{figure}

{The robustness of the ISP method for small values of $D$ is also shown in Figure \ref{Fig_2} (left),
where the diffusion coefficients range from $0$ to $1$.} The figure shows that there is a phase transition around $K_c = D = 1/2$ from coherent to incoherent states, see \cite{ABPRS, GPP} for detailed discussion on the critical value $K_c = D = 1/2$. 

We next fix the strength of noise $D=1$ and varies the coupling strength $K$, and investigate the behavior of $r^\infty(K)$. As expected, the phase transition occurs at $K_c = 2D=2$, see Figure \ref{Fig_2} (right). In the pictures we took a mesh spacing of $0.01$ for $D$ and $K$, in both cases the structure preserving schemes are capable to capture very well the phase transition. As a comparison, we also report the results obtained with the PMC method on a mesh spacing of $0.1$ in $D$ and $K$, where $N_p=5\times 10^4$ particles and $1000$ averages at the steady state are taken. It is clear, that a far superior accuracy can be obtained with the deterministic approach. Let us mention that, even if a careful comparison of the computational cost of the two approaches was outside of the scopes of the present manuscript (we did not try to optimize the coding), for the given parameters the overall computational cost of the PMC method was considerably larger than the ISP method. Beside the different cost of PMC due to the number of particles versus grid points: $O(N_p)$ for PMC and $O(N)$ for ISP for a single value of $K$; we also observed that a smaller time step has to be used at the Monte Carlo level in order to achieve the desired accuracy. 

Finally, let us comment that in order to speed up the computations with the deterministic scheme, we proceed by continuation forward or backward decreasing the value of $D$ or $K$ respectively and taking as initial seed the steady state of the previous computation. This numerical strategy is done in all reported cases of phase transitions below. Using as transition parameter $K$, when continuous phase transitions are expected, as it is the case for identical oscillators, our backward iterations were faster and more accurate to give the right transitions values. However, as we will see in Section \ref{bimodal}, for discontinuous phase transitions, our forward iterations on $K$ by continuation were more accurate. This procedure allows us to continue the bifurcation branches even if their stability basins are becoming very small near the critical order parameters $K$.

Lastly, we numerically investigate the long time behavior of the order parameter with different values of $K$ in Figure \ref{Fig_3} showing that the convergence to steady state gets slower and slower around the critical value $K=2$. This critical slowdown of the convergence near the phase transition critical value is a well-known phenomenon, see \cite{DFL} for instance. It is worth mentioning that in this case using an implicit method is of paramount importance since the explicit CFL condition is too restrictive for such a long time computation close to the critical value.

\begin{figure}[t]
\centering
\includegraphics[scale=0.5]{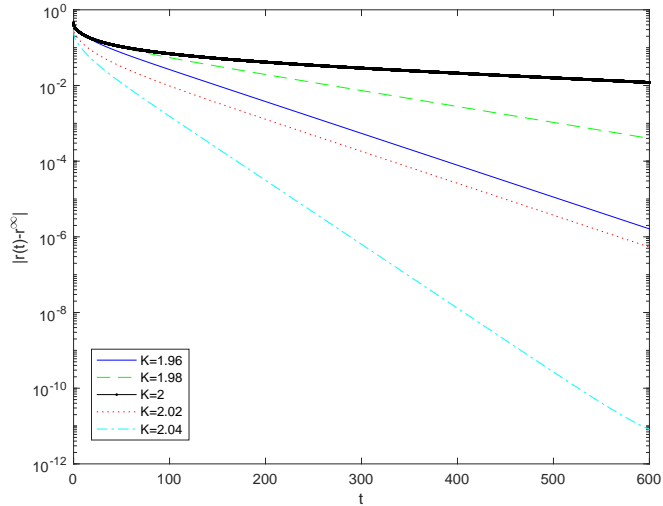}
\caption{Time evolution of the log scale of the steady state $|r(t) - r^\infty|$ with different values of $K$. Here $r^\infty$ is chosen as the value of the order parameter at the final time $T=1200$. We observe the slowing down of the relaxation to the steady state near the critical coupling strength.}
\label{Fig_3}
\end{figure}

\subsection{Nonidentical Kuramoto oscillators} Next, we numerically study the challenging case of the phase transition of the order parameter for nonidentical oscillators. For the initial data, similarly as before, we set
\bq\label{nid_ini}
\rho_0(\theta,\omega) = \frac{c_0}{\sqrt{2\pi\sigma}}\lt(\rho^1 e^{-\frac{1}{2\sigma}\lt(\theta - \frac\pi2 \rt)} + \rho^2 e^{-\frac{1}{2\sigma}\lt(\theta - \frac{3\pi}{2} \rt)} \rt) \quad \mbox{with} \quad \rho^1 = 0.25, \ \rho^2 = 0.75, \ \sigma = 0.1,
\eq
independently of the frequency value $\omega$, 
where $c_0>0$ is again fixed by the normalization, and the nodes $\theta_i$, $i=1,\ldots,N$, are taken equally spaced in $[0,2\pi]$. 

\subsubsection{Uniform distribution}
In this test case, as a function $g(\omega)$ for the natural frequencies, we consider the following uniform distribution centered about $0$ with variance $\sigma_g$:
\[
g(\omega) = \frac{1}{2\sqrt{3\sigma_g}}\mathbf{1}_{\lt[-\sqrt{3\sigma_g}, \sqrt{3\sigma_g} \rt]}(w),  \quad \mbox{with} \quad \sigma_g = 0.1.
\]
We discretize the velocity field using the nodes of the Gauss-Legendre quadrature approximation. In Figure \ref{Fig_NI_uni}, we again take $D=0.5$, $N=200$, $M=10$ to investigate the phase transition of the order parameter $r^\infty(K)$ numerically. Note that in this case, the critical coupling strength $K_c(D)$ can be computed from the formula \cite{Saka} 
\bq\label{cri_val}
K_c(D) = 2\lt(\int_\R \frac{g(D\omega)}{\omega^2 + 1}\,d\omega \rt)^{-1}.
\eq
The expression above is derived under the symmetry assumption on the distribution function $g$ around its mean value. The critical coupling strength in our case is $K_c \approx 1.31836$. We took different mesh spacing in $K$, in the interval $[1.25, 1,45]$ we use a step of $0.005$ whereas in $[1.31,1.33]$ a step $5\times 10^{-4}$. It is known that the order parameter $r^\infty(K)$ is proportional to $\sqrt{K-K_c}$ near and above the critical coupling strength, see \cite{Saka} for instance. In order to confirm that numerically, in the small figure in Figure \ref{Fig_NI_uni}, we compare between $r^\infty(K)$ and $c\sqrt{K-K_c}$, where the constant $c$ is computed by using the built-in {\it polyfit} MATLAB command. As a comparison, we also report the results obtained with the PMC method on a mesh spacing of $0.1$ in $K$, where $N_p=5\times 10^5$ particles and $10000$ averages at the steady state are taken. Even in the case of non identical oscillators, the deterministic ISP scheme, which has a cost of $O(M N)$ against $O(N_p)$ for a single value of $K$, was far more efficient than the corresponding PMC.

\begin{figure}[ht!]
\protect\includegraphics[scale=0.5]{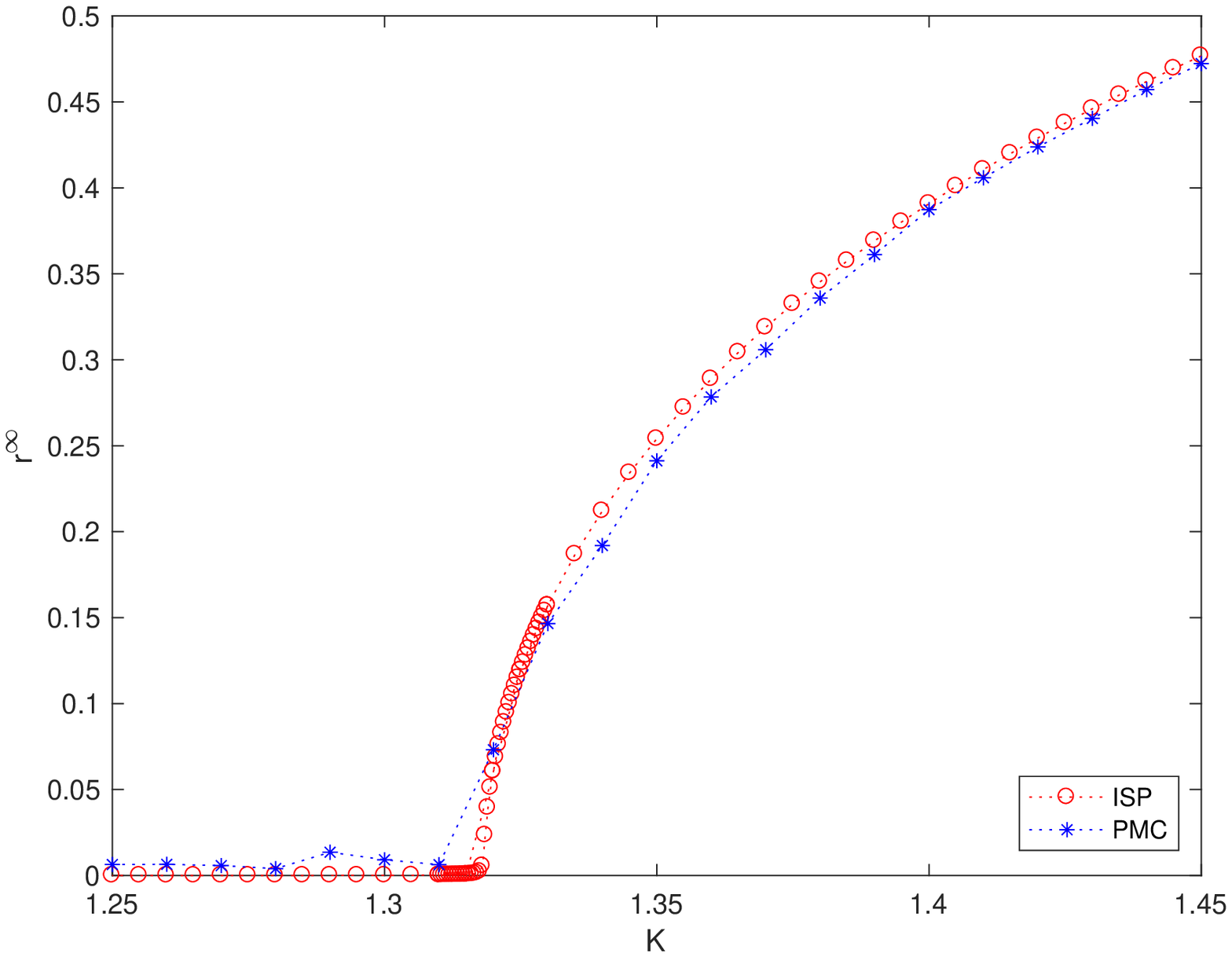}
\llap{\shortstack{%
        \includegraphics[scale=.19]{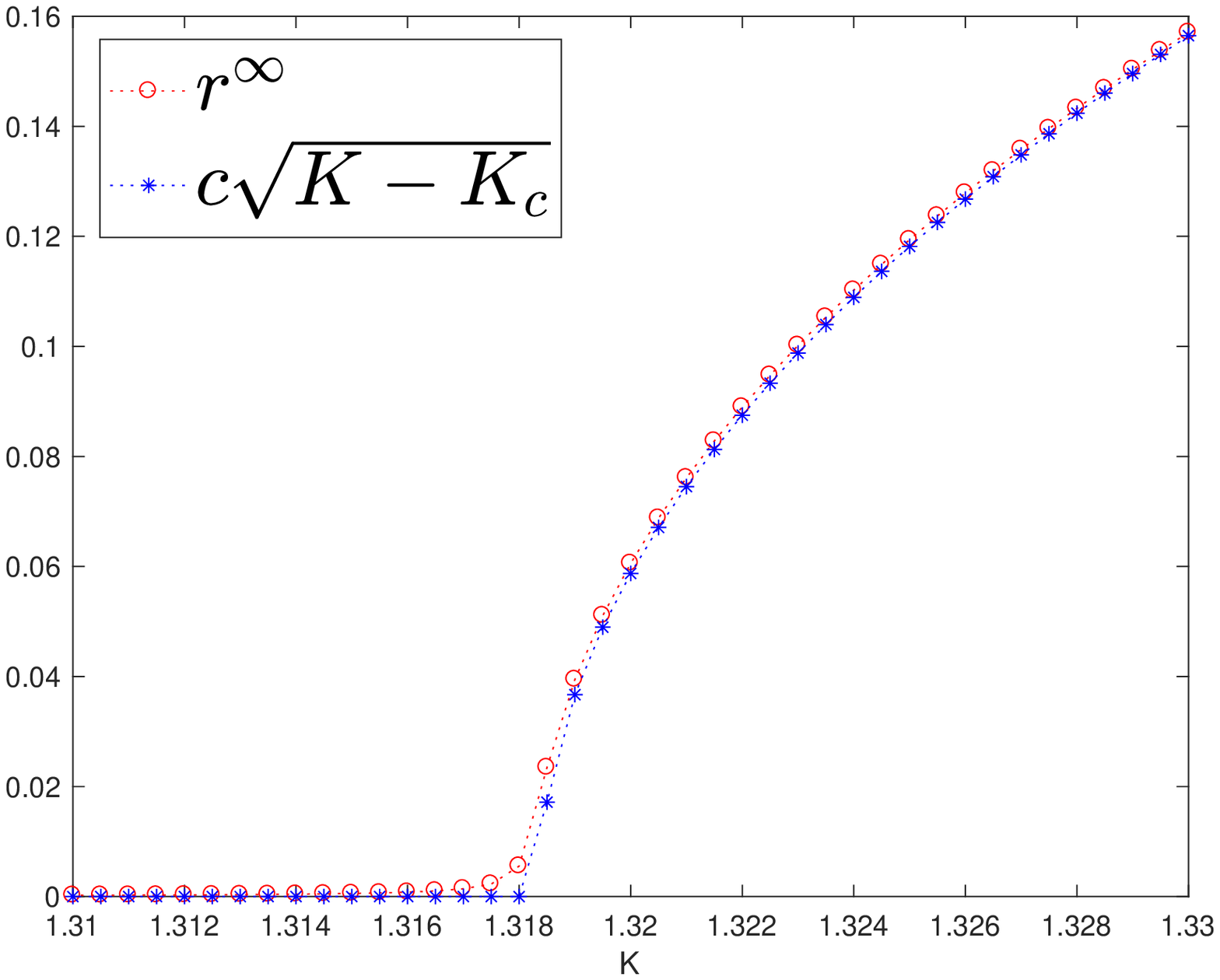}\\
        \rule{0ex}{1.5in}%
      }
  \rule{1.89in}{0ex}}
  \caption{Uniform distribution case: phase transition of the order parameter $r^\infty(K)$. ISP scheme for $N=200$, $M=10$. PMC method with $5\times 10^5$ particles and $10000$ averages.}
  \label{Fig_NI_uni}
\end{figure}

\begin{figure}[t]
\centering
\includegraphics[scale=0.38]{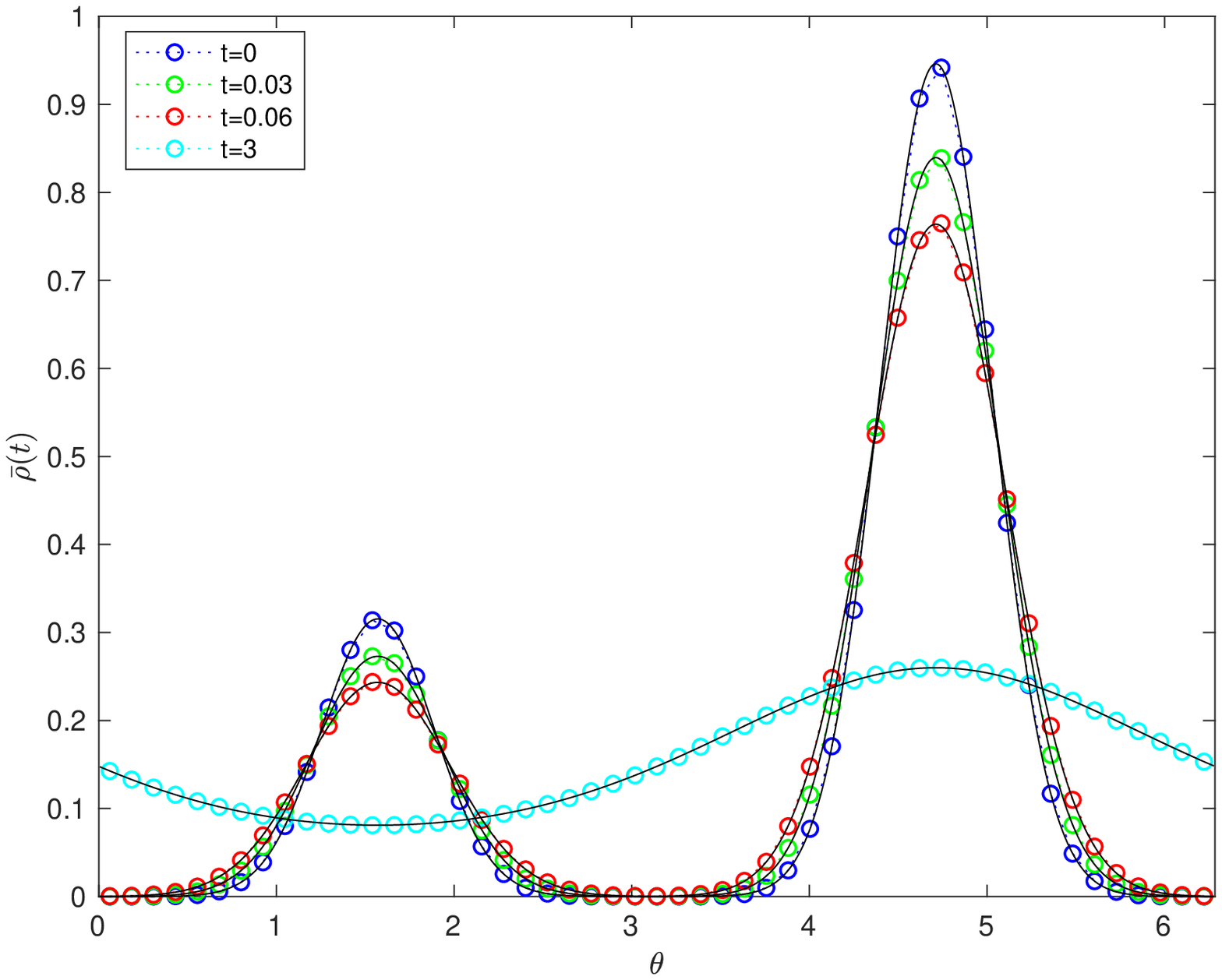}
\includegraphics[scale=0.38]{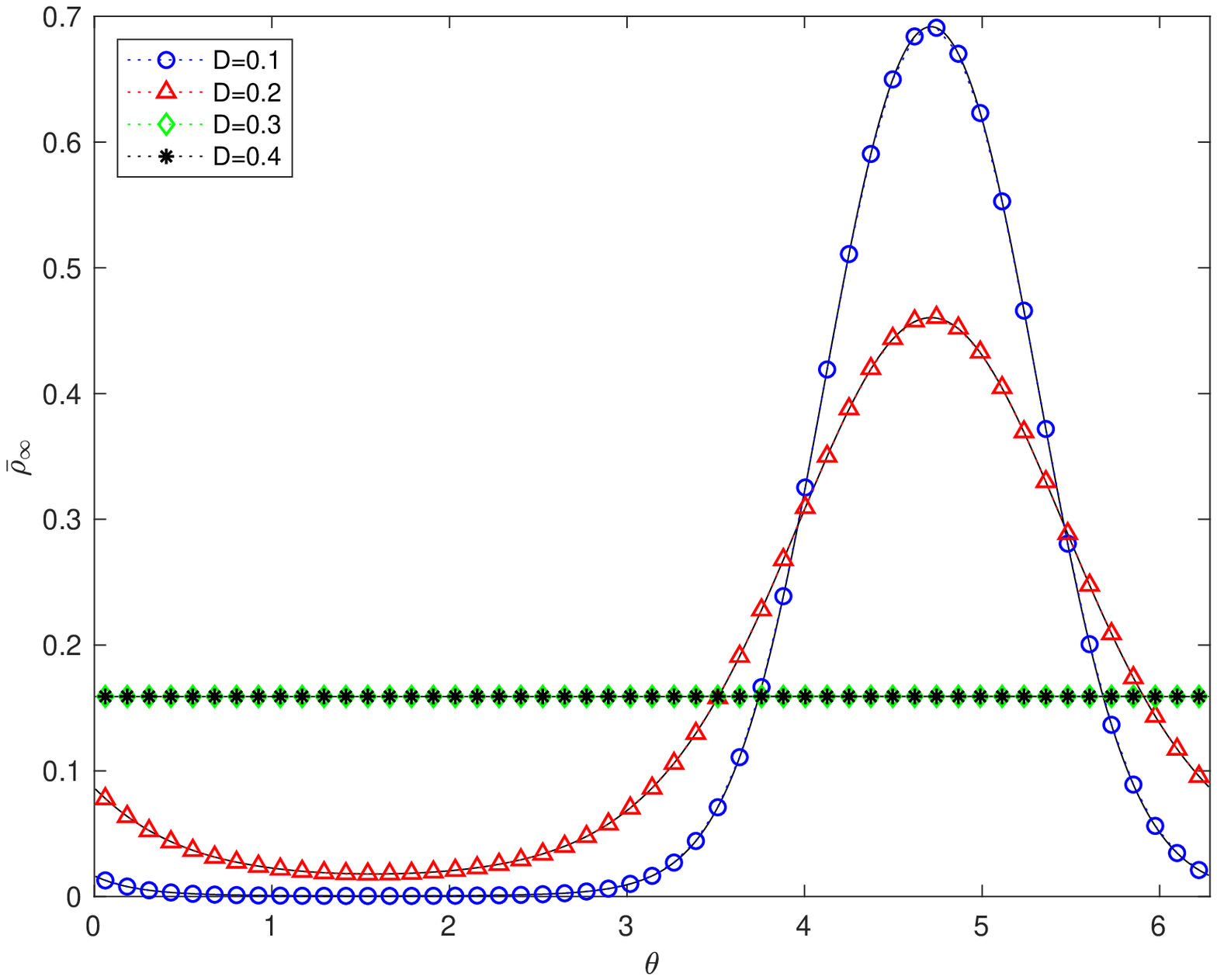}
\caption{Uniform distribution: time evolution of the averaged solution $\bar{\rho}(t)$ with $N=51$, $M=10$, $K=1$, and $D=0.5$ (left). Evolution of the averaged steady state $\bar\rho_\infty$ with $N=51$, $M=10$, $K=1$ with different values of $D$ (right). 
}
\label{Fig_rho_uni}
\end{figure}

We also investigate the time evolution of the solution in Figure \ref{Fig_rho_uni}  (left) where we reported at various times the behavior of the average value of the solution 
\[
\bar\rho(\theta,t) := \int_\R \rho(\theta,\omega,t)g(\omega)\,d\omega,
\]
for $N=51$, $M=10$, $K=1$, and $D=0.5$. The evolution of the averaged steady state $\bar\rho_\infty$ with $D = k/10$, $k=1,\cdots,4$ is demonstrated in Figure \ref{Fig_rho_uni} (right). 

In Figure \ref{Fig_steady_uni}, we show the steady state for the density $\rho_\infty(\vartheta,\omega)$ for $D=0.5$, $N=100$, and $M = 30$ with different values of $K$. We choose $K = 0.5$, which is in the region of subcritical coupling strength and $K=2$ in the region of supercritical one.

\begin{figure}[t]
\centering
\includegraphics[scale=0.35]{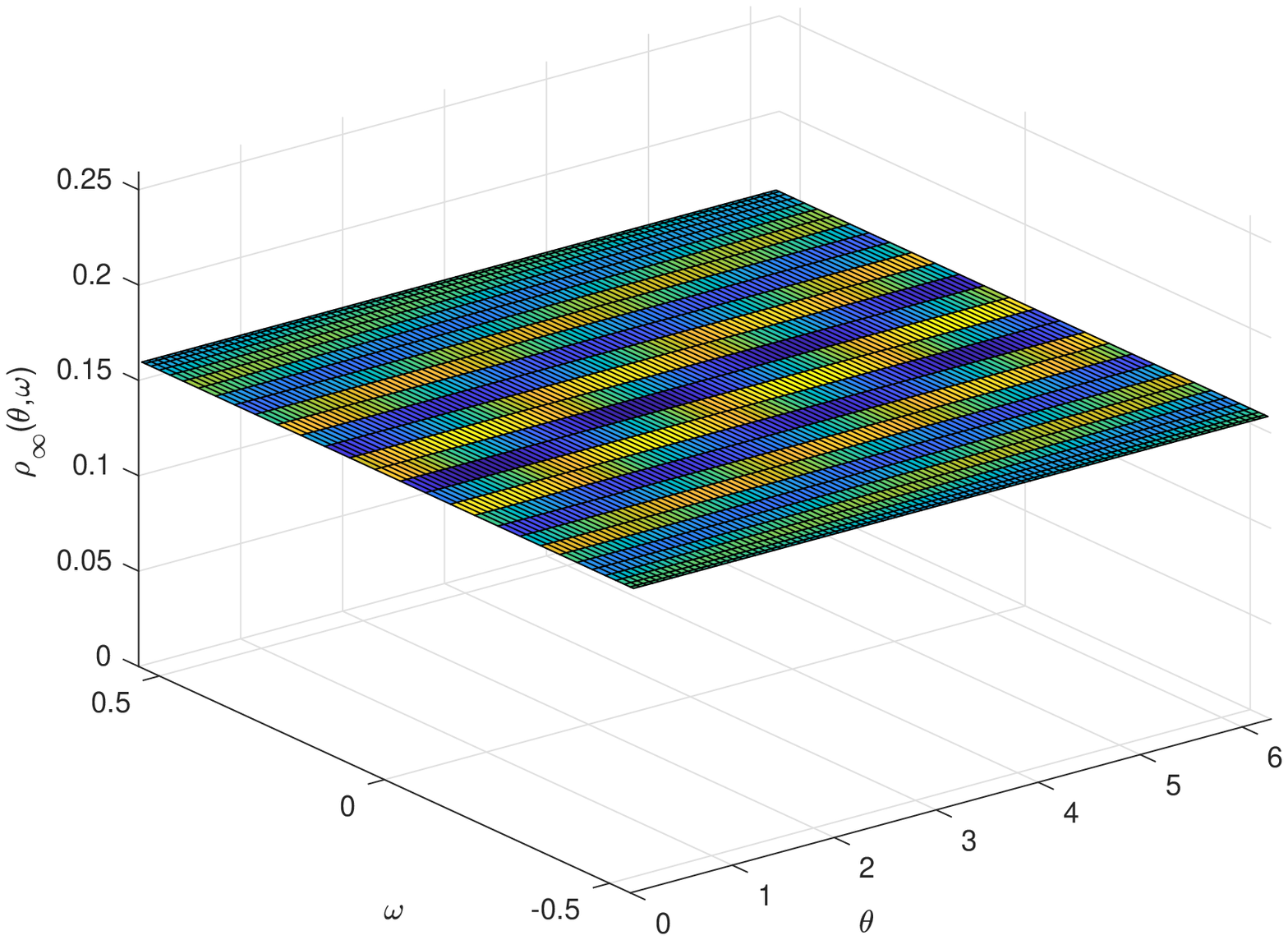}
\includegraphics[scale=0.35]{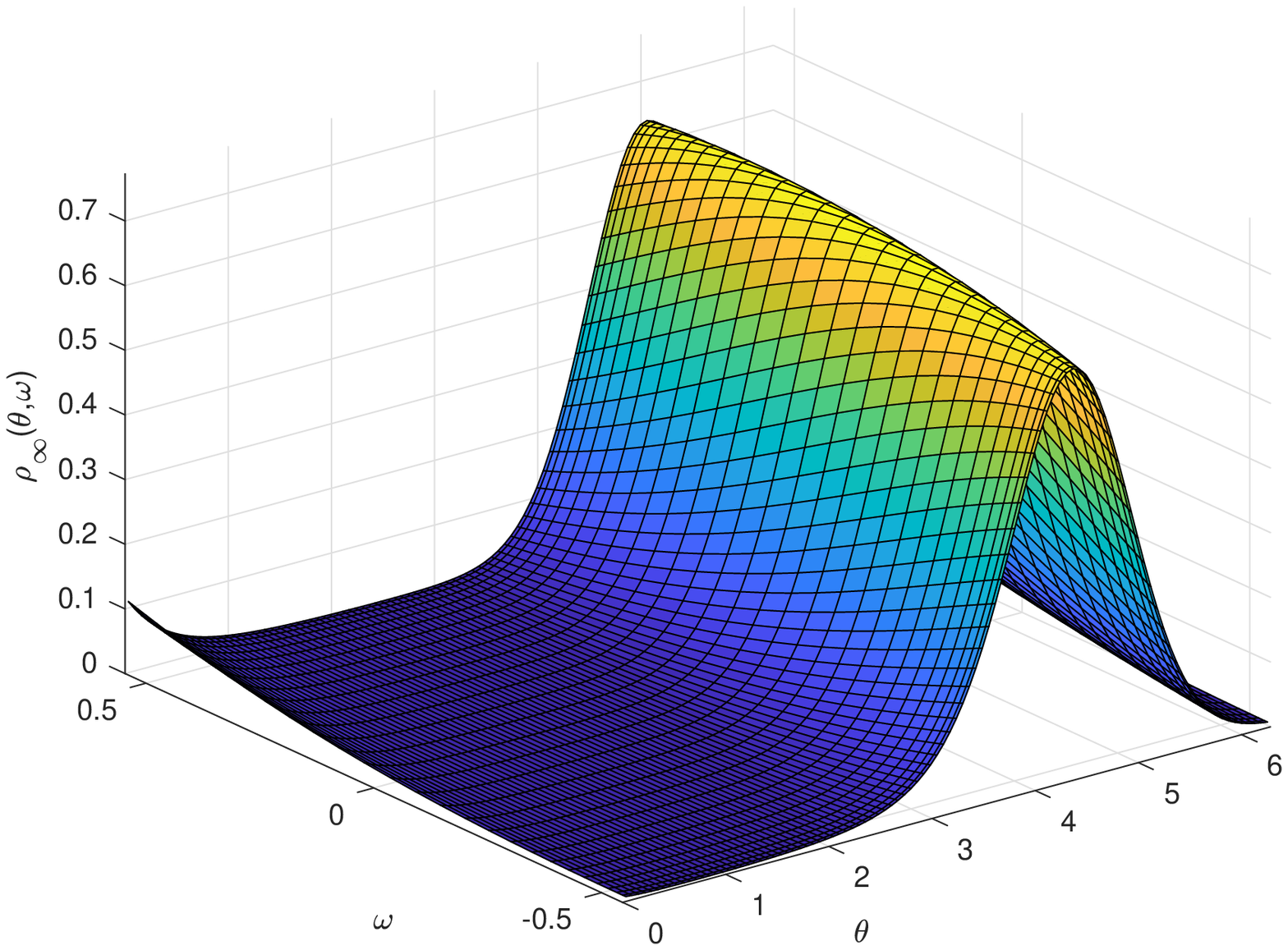}
\caption{Uniform distribution: steady state for the density $\rho_\infty(\vartheta,\omega)$ with $N=100$, $M=30$ in the subcritical case with $K=0.5$ and $D=0.5$ (left) and in the supercritical case with $K=2$ and $D=0.5$ (right)} 
\label{Fig_steady_uni}
\end{figure}

\begin{figure}[ht]
\protect\includegraphics[scale=0.5]{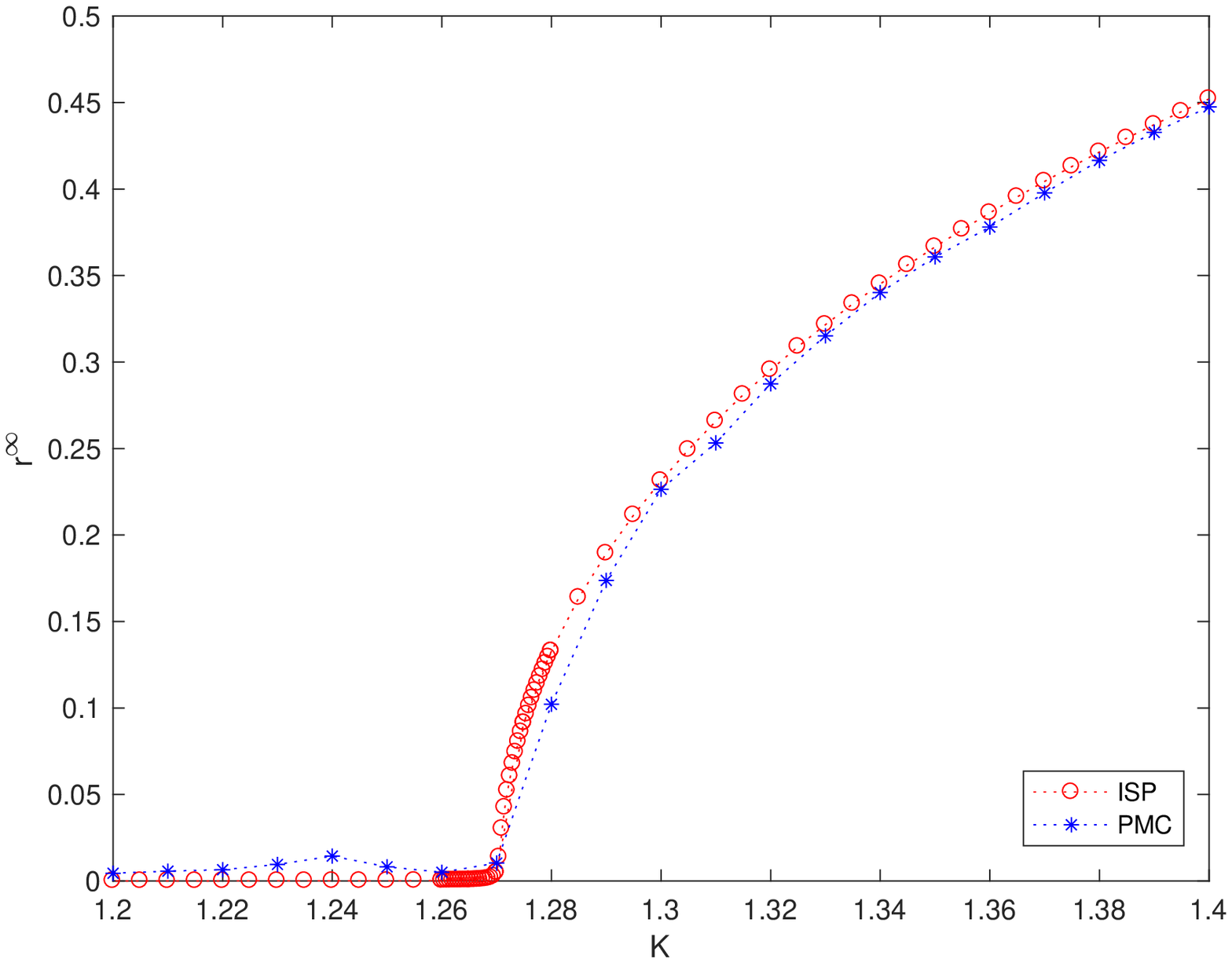}
\llap{\shortstack{%
        \includegraphics[scale=.2]{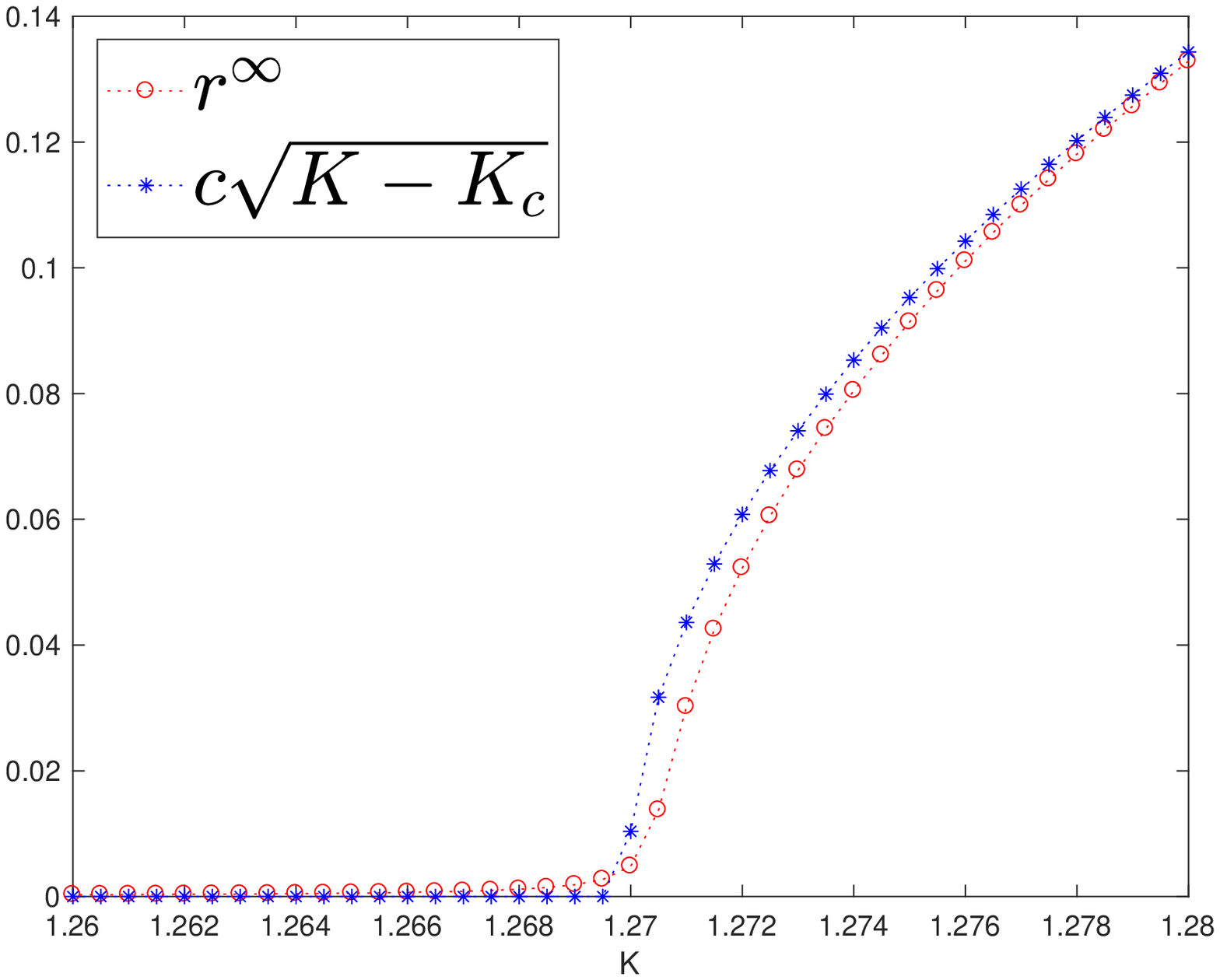}\\
        \rule{0ex}{1.47in}%
      }
  \rule{1.82in}{0ex}}
  \caption{Gaussian distribution case: phase transition of the order parameter $r^\infty(K)$. ISP scheme for $N=200$, $M=10$. PMC method with $5\times 10^5$ particles and $10000$ averages.}
  \label{Fig_NI_gau}
\end{figure}

\subsubsection{Gaussian distribution}\label{ssec_gau} Next, we choose the function $g(\omega)$ for the natural frequencies as the following Gaussian distribution centered in $0$ with variance $\sigma_g$:
\[
g(\omega) = \frac{1}{\sqrt{2\pi\sigma_g}}  e^{-\frac{\omega^2}{2\sigma_g}} \quad \mbox{with} \quad \sigma_g = 0.1.
\]
We compute the definite integral in the velocity field \eqref{kku}, see also \eqref{eq_app_vel}, using the Gauss-Hermite quadrature approximation, therefore the $M$ grid points for the natural frequencies corresponds to the nodes of such quadrature formula. In Figure \ref{Fig_NI_gau}, we take $D=0.5$, $N=200$, $M=10$ to investigate the phase transition of the order parameter $r^\infty(K)$ numerically. As mentioned before, the formula for the critical coupling strength $K_c$ in \eqref{cri_val} is also valid in the Gaussian case, and it gives the value $K_c\simeq 1.26994$ in our setting. Similarly, we took a mesh spacing in $K$ of $0.005$ for the interval $[1.2,1.4]$, and we took a finer spacing of $5\times 10^{-4}$ for the interval $[1.26,1.28]$, where the phase transition of the order parameter is expected to happen. The comparison between $r^\infty(K)$ and $c\sqrt{K - K_c}$ is shown in the zoomed image in Figure  \ref{Fig_NI_gau}. Again, we also report the results obtained with the PMC method on a mesh spacing of $0.1$ in $K$, where $5\times 10^5$ particles and $10000$ averages at the steady state are taken. Next, in Figure \ref{Fig_rho_gau} we plot the time evolution of $\bar\rho(t)$ and the evolution $\bar\rho_\infty(D)$ with different values of $D$. 

\begin{figure}[t]
\centering
\includegraphics[scale=0.38]{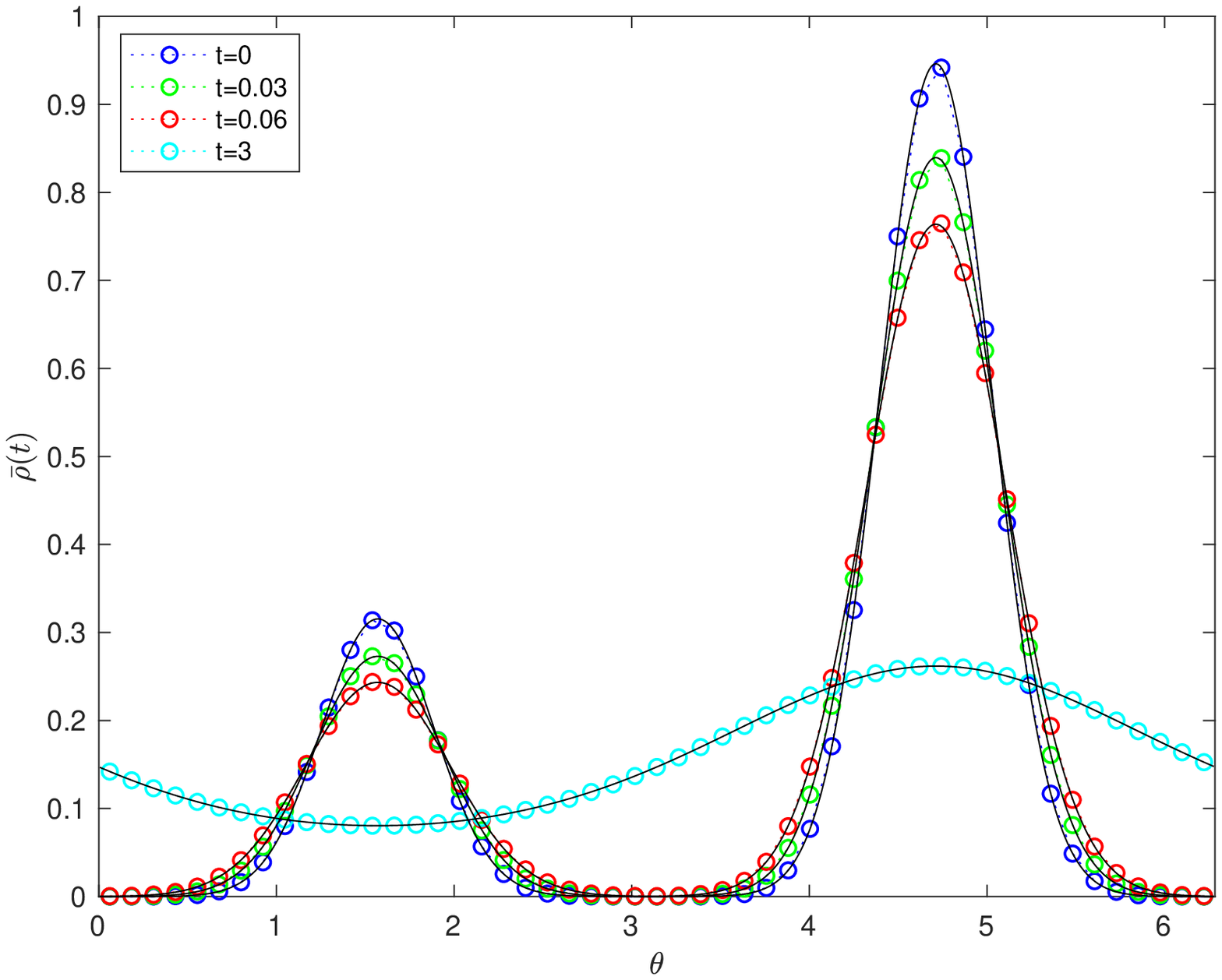}
\includegraphics[scale=0.38]{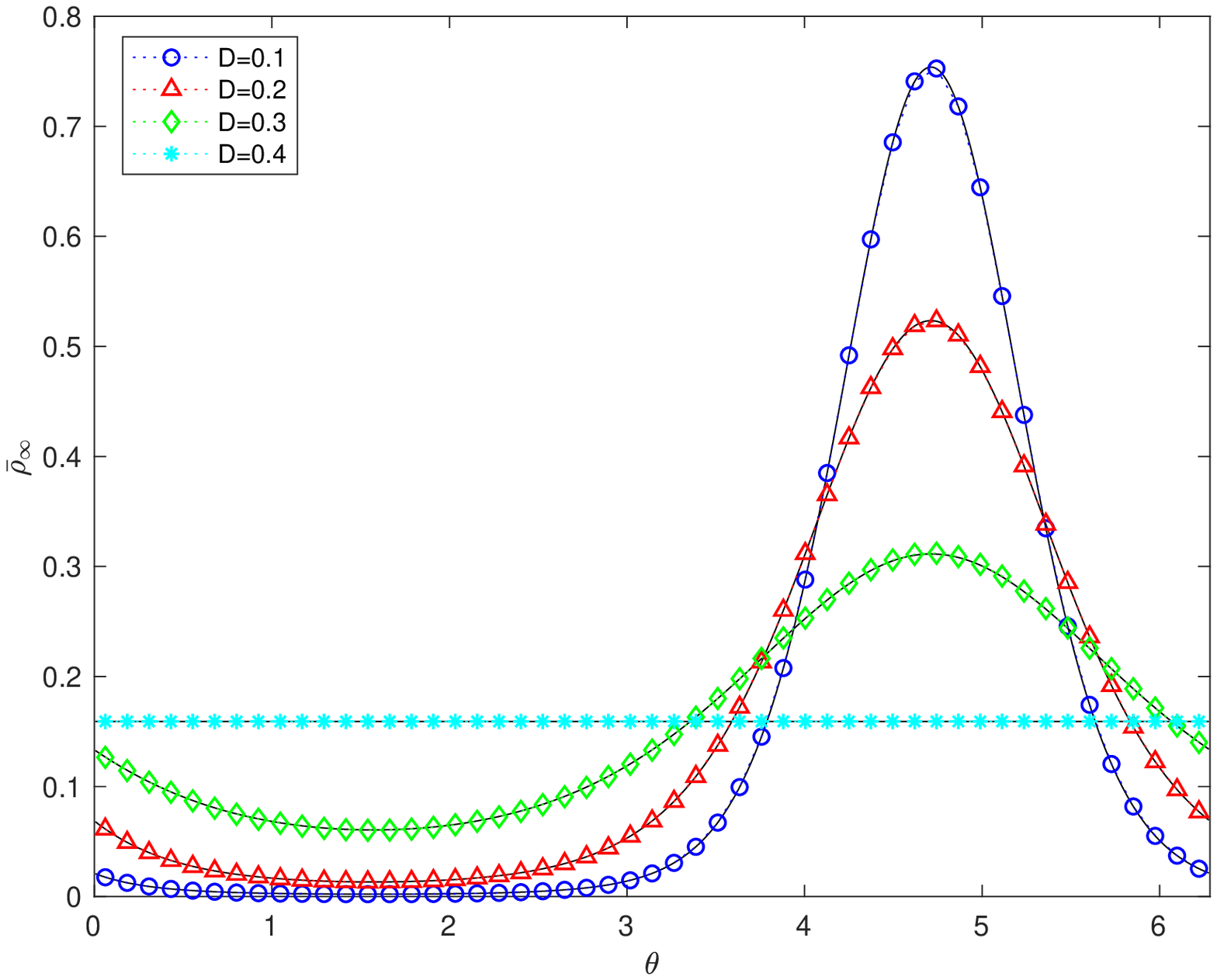}
\caption{Gaussian distribution: time evolution of the averaged solution $\bar{\rho}(\cdot,t)$ with $N=51$, $M=10$, $K=1$, and $D=0.5$ (left). Evolution of the averaged steady state $\bar\rho_\infty$ with respect to $g(\omega)$ with $N=51$, $M=10$, $K=1$ with different values of $D$ (right). }
\label{Fig_rho_gau}
\end{figure}

\begin{figure}[ht]
\centering
\includegraphics[scale=0.38]{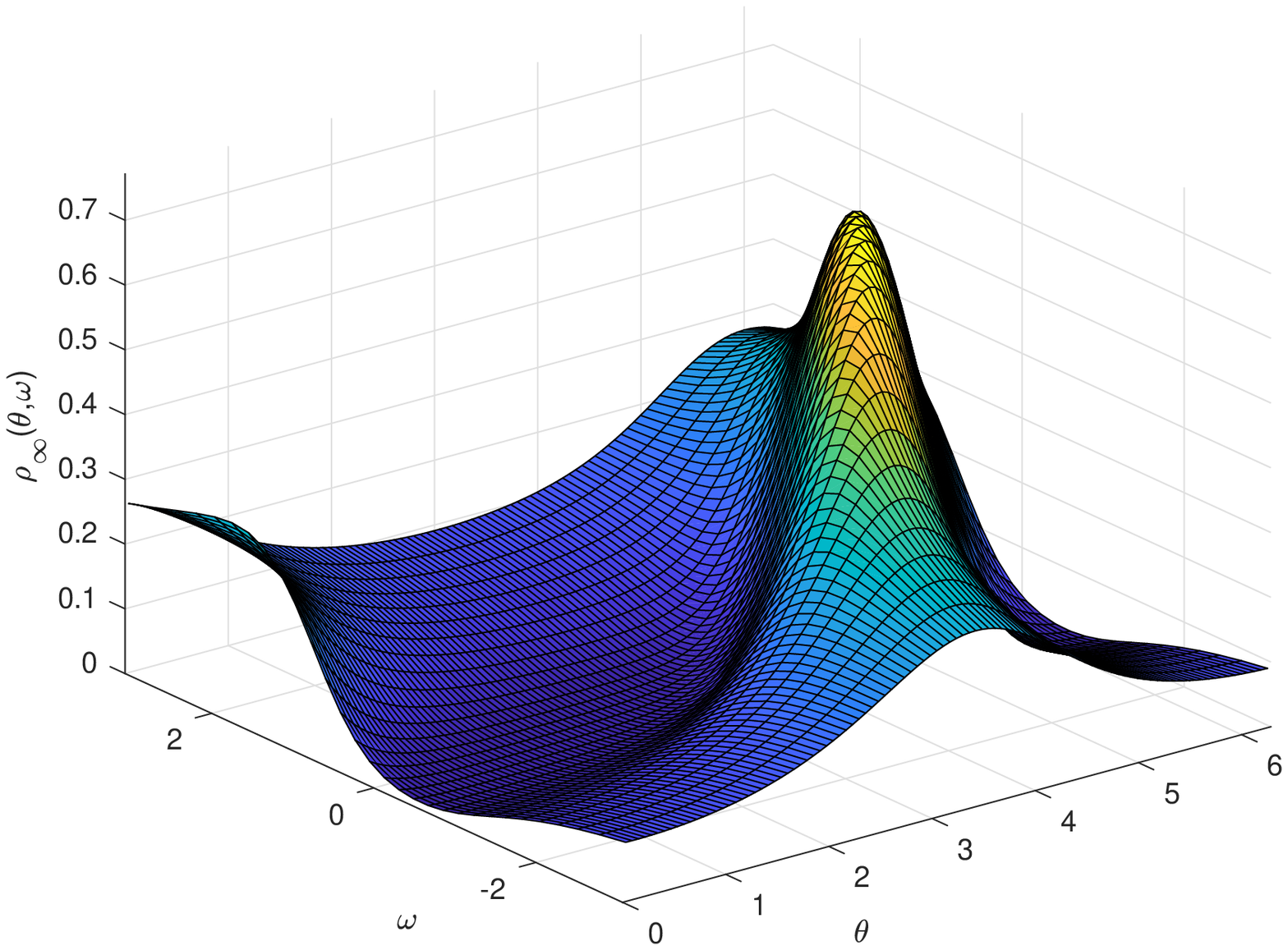}
\includegraphics[scale=0.38]{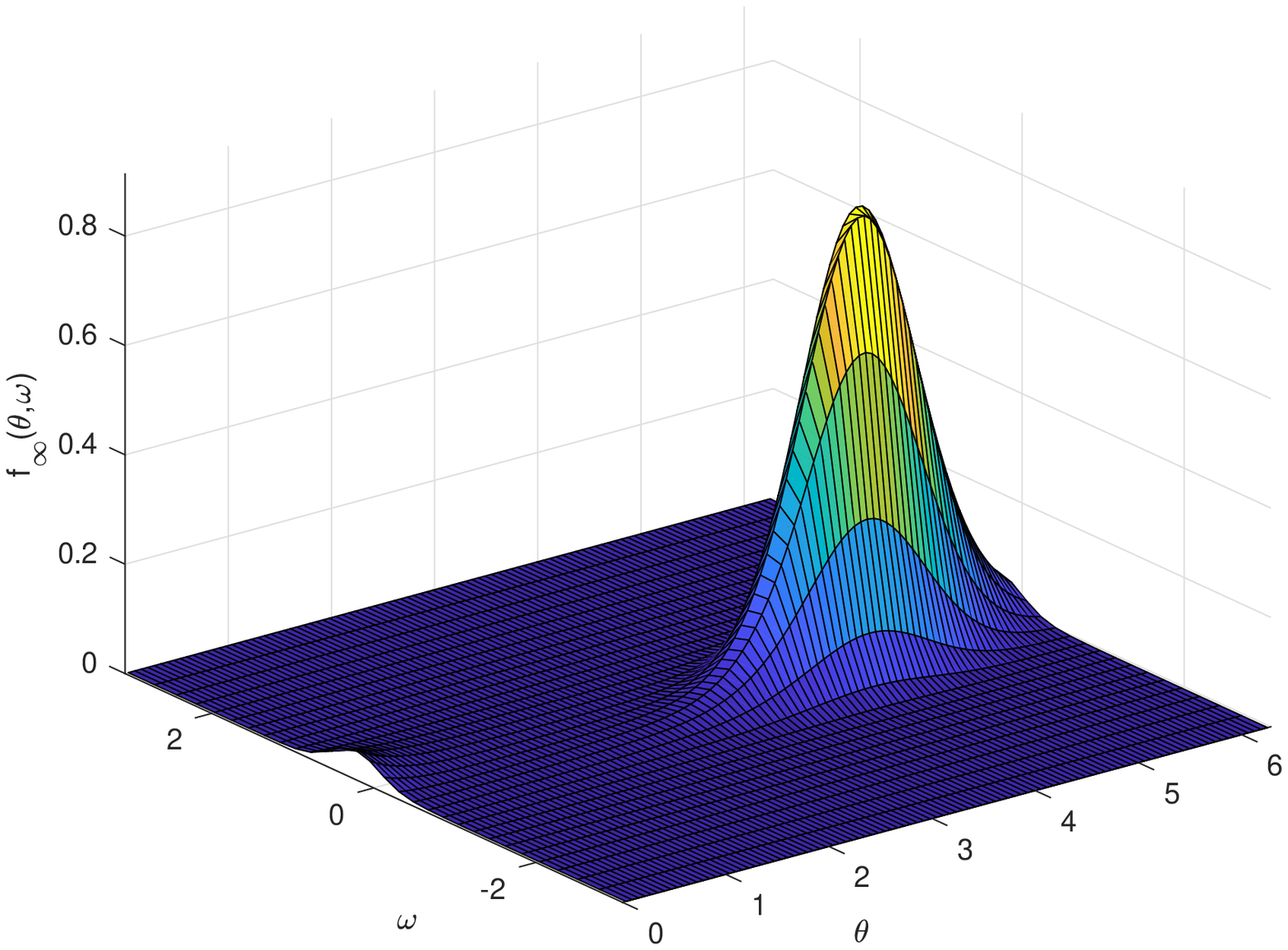}
\caption{Gaussian distribution: steady state for the density $\rho_\infty(\vartheta,\omega)$ (left)  and $f_\infty(\vartheta,\omega)$ (right) with $K=2$ and $D=0.5$.} 
\label{Fig_steady_gau}
\end{figure}

Finally, in Figure \ref{Fig_steady_gau}, we show the steady state $\rho_\infty(\vartheta,\omega)$ and $f_\infty(\vartheta,\omega)$, which is the weighted kinetic density introduced in \eqref{fff}, for $K=2$, $D=0.5$, $N=100$, and $M = 30$, where the coupling strength $K=2$ lies in the region of supercritical coupling strength. Clearly, if the coupling strength $K$ lies in the subcritical region, we have the constant state $\rho_\infty$ as in the steady state of Figure \ref{Fig_steady_uni} (left), so we omit the result.

\subsubsection{Bimodal distribution} 
\label{bimodal}
In the last example, we consider a double Gaussian (bimodal) distribution function $g(\omega)$ for the natural frequencies:
\[
g(\omega) = \frac{1}{2\sqrt{2\pi\sigma_g}}  e^{-\frac{(\omega - \mu)^2}{2\sigma_g}} + \frac{1}{2\sqrt{2\pi\sigma_g}}  e^{-\frac{(\omega + \mu)^2}{2\sigma_g}} \quad \mbox{with} \quad \mu = \frac{\sqrt{2} + 1}{4\sqrt{2}} \quad \sigma_g = 0.001.
\]
Note that, in this case, it is observed that a discontinuous phase transition of the order parameter occurs, see \cite{BNS,Craw94}. Similarly as before, we investigate this discontinuous phase transition $r^\infty(K)$ with $N=200$, $M=10$, and $D=0.5$ in Figure \ref{Fig_NI_bimodal}. We took different mesh spacing in $K$, we use a step of $0.005$ in the interval $[1.5,2]$, a step of $0.0005$ in the zoomed image in Figure \ref{Fig_NI_bimodal}. Again, we also put the PMC method result with a mesh spacing of $0.05$ in the interval $[1.5,2]$ and 0.002 in the zoomed image, where $5\times 10^4$ particles and $10000$ averages at the steady state are taken. The critical coupling strength $K_c$ is given by $K_c \simeq 1.72584$. As mentioned before, in this bimodal case, the forward ISP method is capable to describe very well the discontinuity in the phase transition. Note that, the backward ISP and the PMC methods fail to capture the correct jump location in the discontinuous phase transition. This might be due to the propagation of numerical errors leading to a jump outside the stability basin of the homogeneous state that is getting smaller and smaller as one approaches from the left the critical value of $K$. This is a strong indication of hysteresis phenomena usually characteristic of discontinuous phase transitions since both the homogeneous state and the second bifurcating branch have small basins of attraction coexisting in a range of the order parameter $K$.

\begin{figure}[ht]
\protect\includegraphics[scale=0.5]{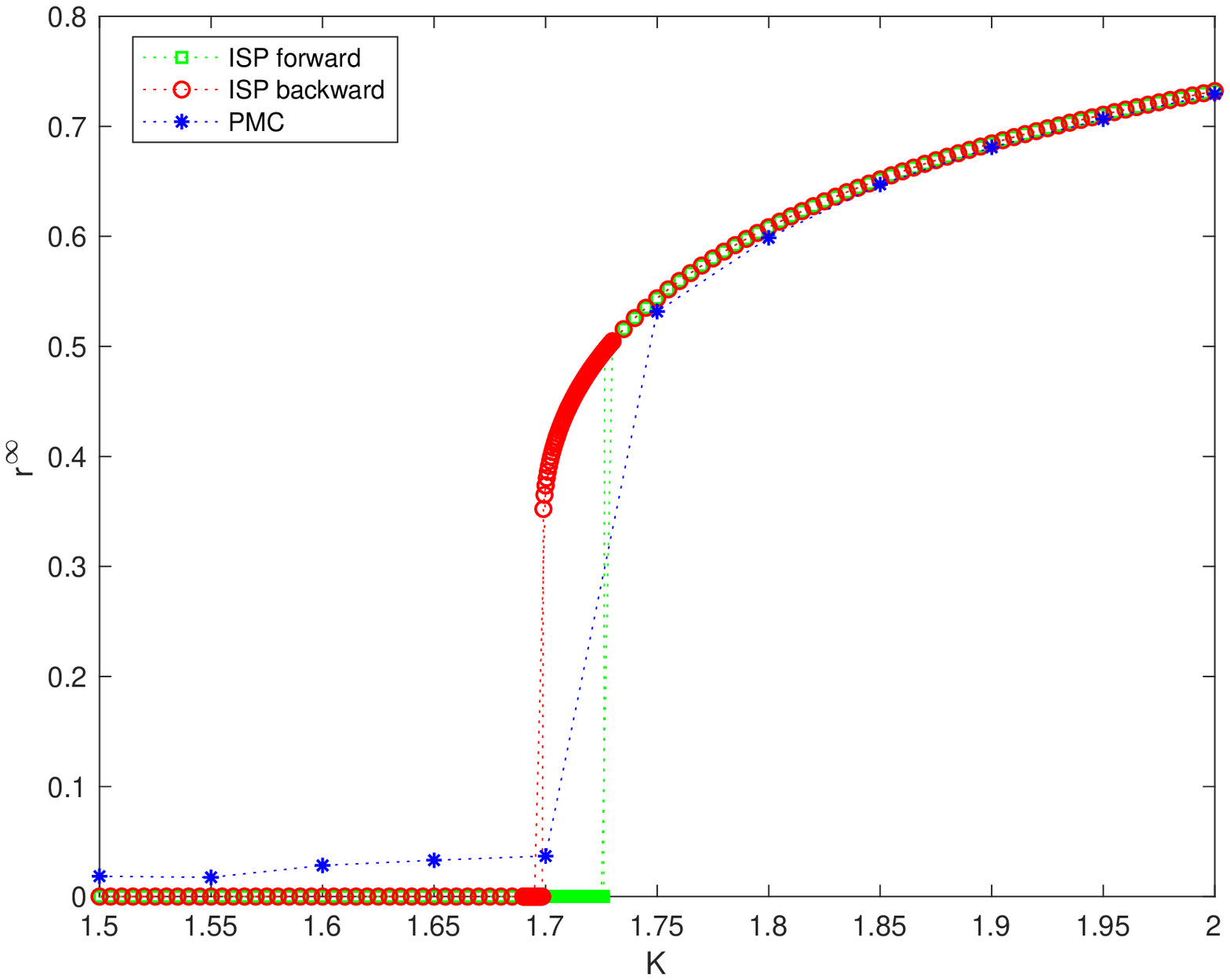}
\llap{\shortstack{%
        \includegraphics[scale=.205]{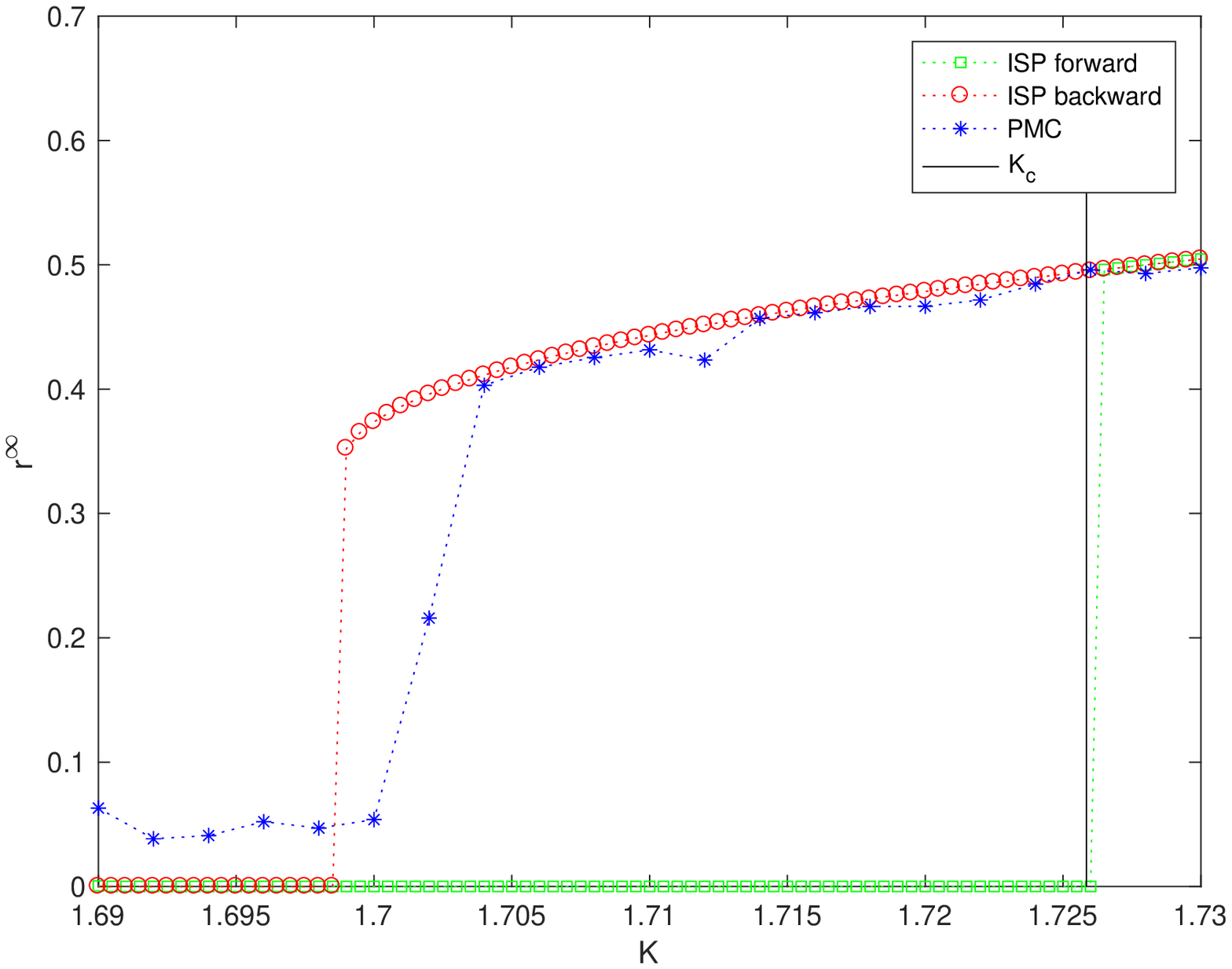}\\
        \rule{0ex}{0.5in}%
      }
  \rule{0.4in}{0ex}}
  \caption{Bimodal distribution case: phase transition of the order parameter $r^\infty(K)$. ISP scheme for $N=200$, $M=10$. PMC method with $5\times 10^4$ particles and $10000$ averages.}
  \label{Fig_NI_bimodal}
\end{figure}

Similarly as before, we also show the time evolution of $\bar\rho(t)$, the evolution $\bar\rho_\infty(D)$ with different values of $D$ in Figure \ref{Fig_rho_bi}, and the steady states $\rho_\infty(\theta,\omega)$ and $f_\infty(\theta,\omega)$ for $K=2$, $D = 0.5, N=100$, and $M=30$ in Figure \ref{Fig_steady_bi}. Note that it follows from Figure \ref{Fig_NI_bimodal} that the value $K=2$ lies in the region of supercritical coupling strength.

\begin{figure}[t]
\centering
\includegraphics[scale=0.38]{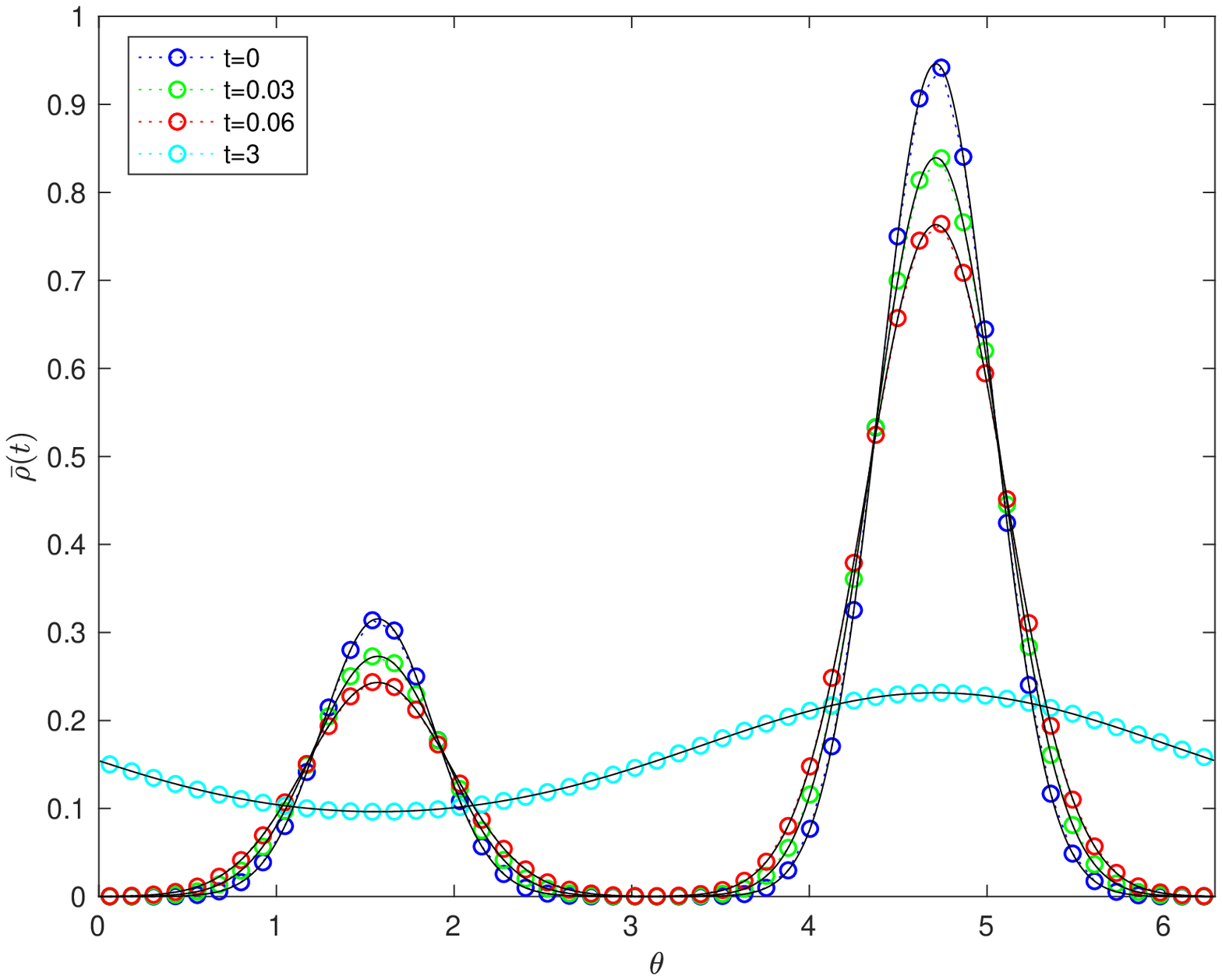}
\includegraphics[scale=0.38]{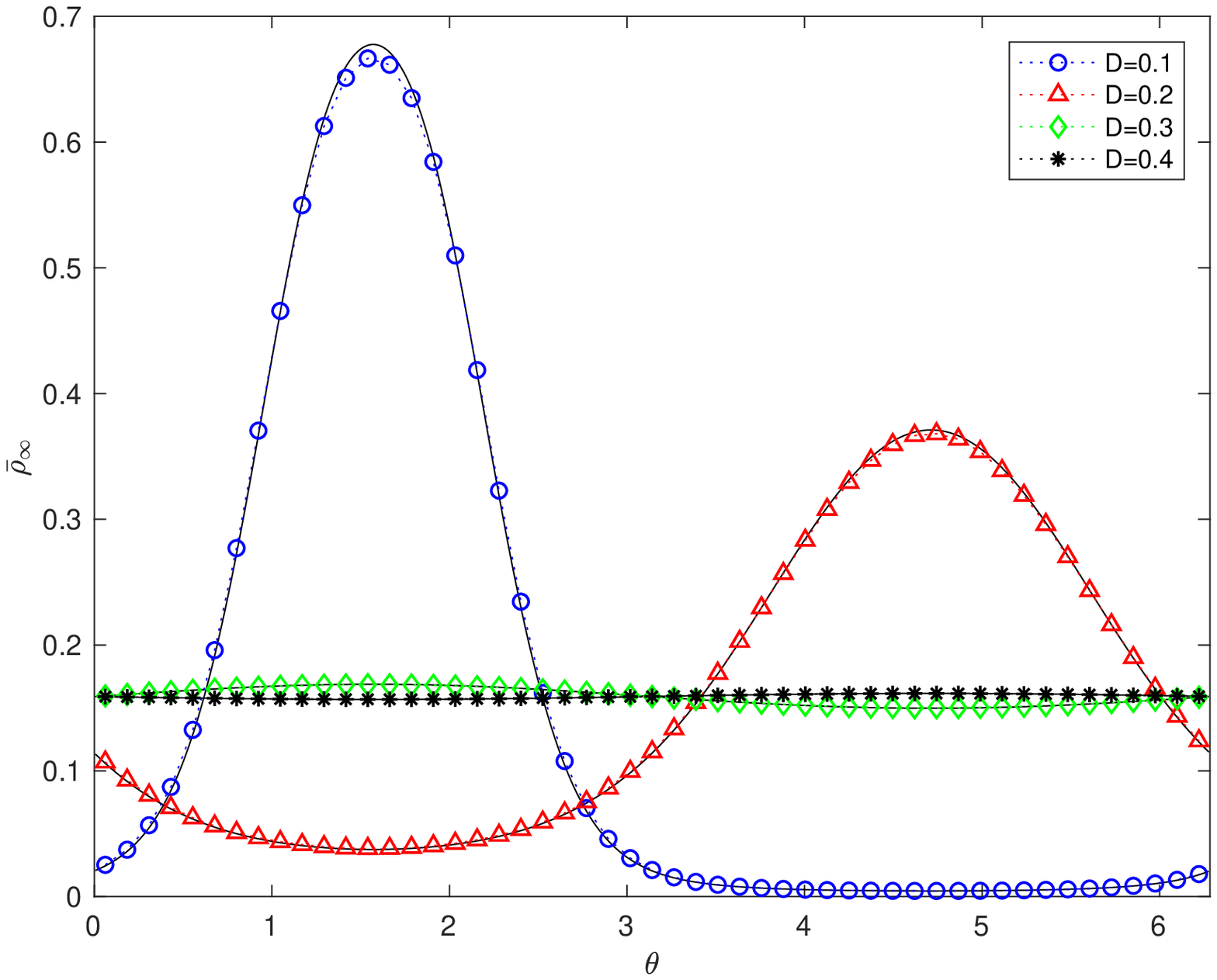}
\caption{Bimodal distribution: time evolution of the averaged solution $\bar{\rho}(t)$ with $N=51$, $M=10$, $K=1$, and $D=0.5$ (left). Evolution of the averaged steady state $\bar\rho_\infty$ with $N=51$, $M=10$, $K=1$ with different values of $D$ (right). }
\label{Fig_rho_bi}
\end{figure}

\begin{figure}[ht]
\centering
\includegraphics[scale=0.38]{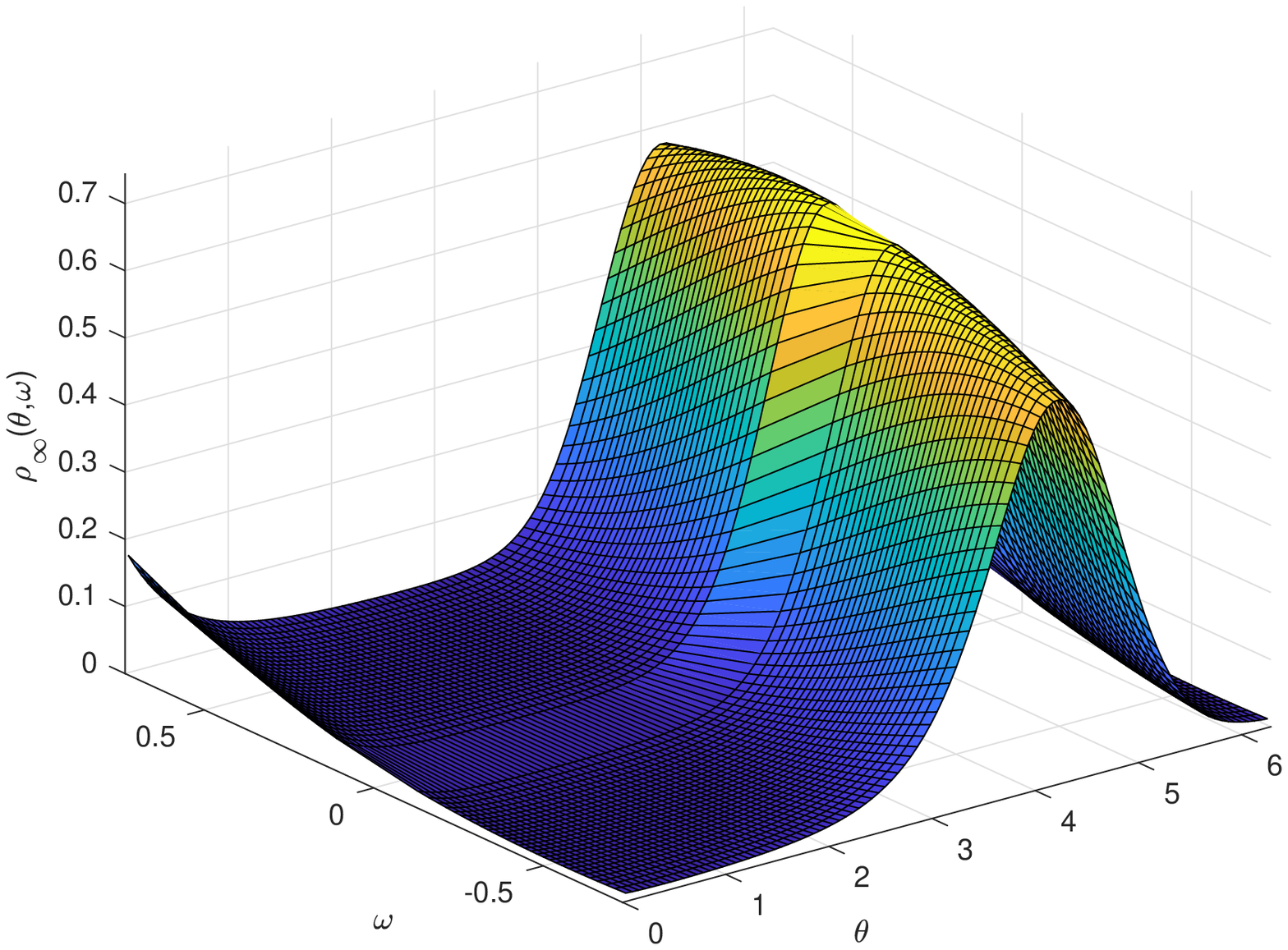}
\includegraphics[scale=0.38]{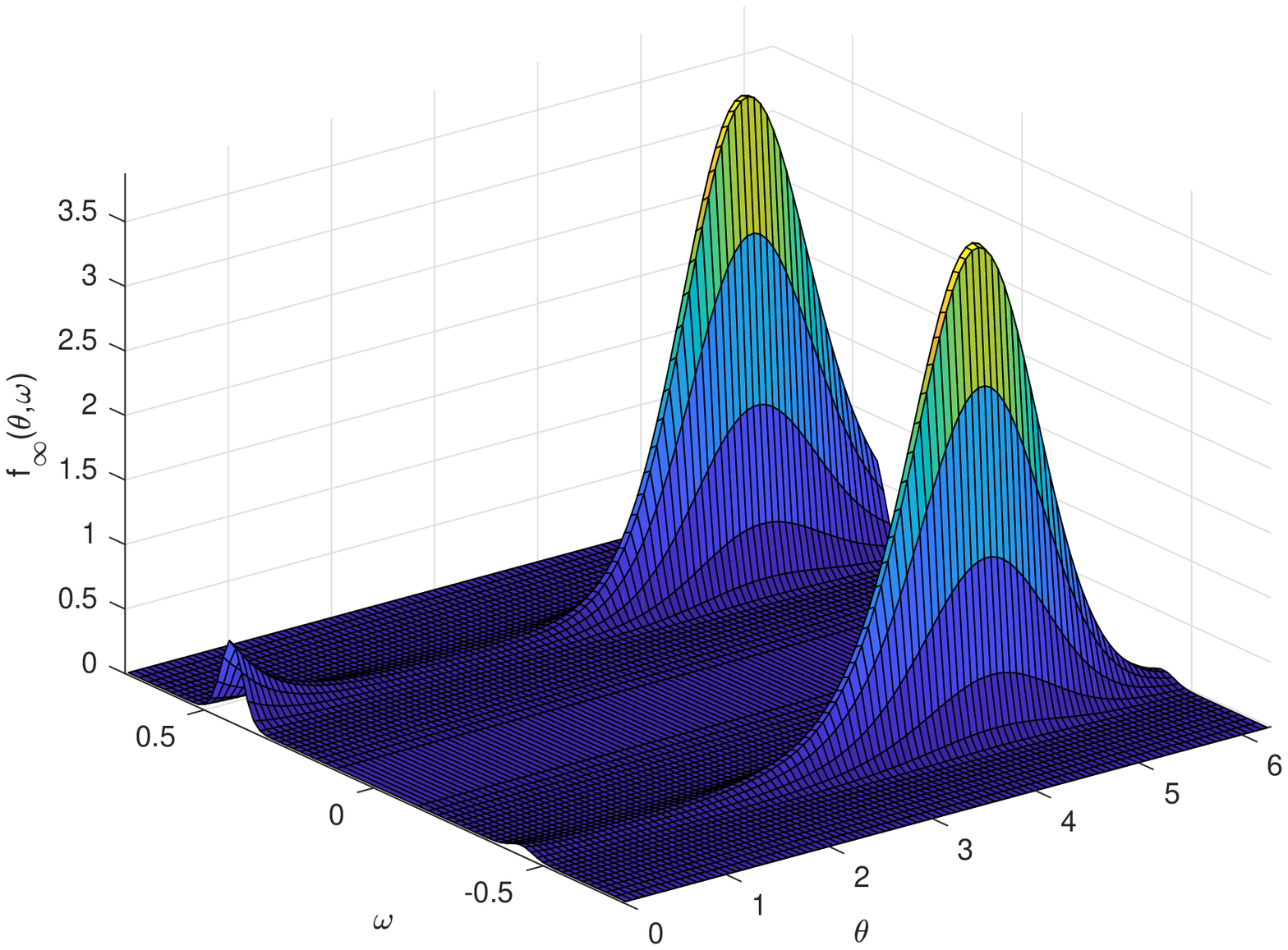}
\caption{Bimodal distribution: steady state for the density $\rho_\infty(\vartheta,\omega)$ (left)  and $f_\infty(\vartheta,\omega)$ (right) with $K=2$ and $D=0.5$.} 
\label{Fig_steady_bi}
\end{figure}

{
\subsection{Kuramoto-Daido type model} In this part, we use our numerical scheme for the Kuramoto-Daido type model, whose coupling function includes higher harmonic terms such as $\sin 2(\theta)$. More precisely, we consider the following continuum Kuramoto-Daido type equation:
\begin{align}
\begin{aligned} \label{ku-dai}
\displaystyle &\partial_t \rho + \partial_{\theta} (\tilde u[\rho] \rho) = D\pa^2_\theta\rho, \qquad (\theta, \omega) \in \T \times \R,~~t > 0, \\
& \tilde u[\rho](\theta, \omega, t) = \omega + K \int_{\T \times \R}\lt(\sin(\theta_* - \theta) + h\sin(2(\theta_* - \theta))\rt)\rho(\theta_*, \omega, t) g(\omega)\,d\theta_* d\omega,
\end{aligned}
\end{align}
It is known that the above system \eqref{ku-dai} with $h \in (0,1)$ exhibits a hysteresis phenomenon. We refer to \cite{Chi, CN11} for a description of the synchronization transition as a bifurcation problem. In order to analyze the phase transition as a function of the order parameter, we set the same initial data with \eqref{nid_ini} and take the Gaussian distribution $g(\omega)$ defined in Section \ref{ssec_gau}. 

\begin{figure}[ht]
\protect\includegraphics[scale=0.4]{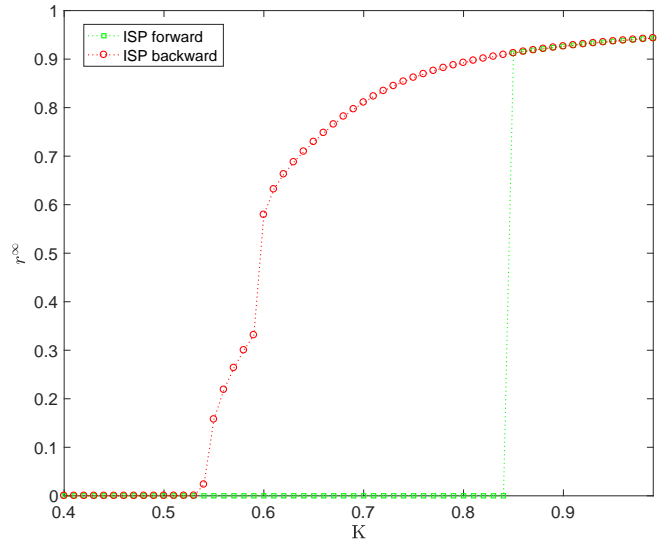}
  \caption{Kuramoto-Daido type model with Gaussian distribution: phase transition of the order parameter $r^\infty(K)$. }
  \label{Fig_NI_gau_ku_dai}
\end{figure}

In Figure \ref{Fig_NI_gau_ku_dai}, we take $D=0.1$, $N=200$, $M=10$, $h=0.5$ to investigate the phase transition of the order parameter $r^\infty(K)$ numerically. We also take a mesh spacing in $K$ of $0.01$. As observed in \cite{CN11}, we find a different type of bifurcation with respect to Figure \ref{Fig_NI_gau}. Figure \ref{Fig_NI_gau_ku_dai} shows an explosive jump from the incoherent state to the coherent one when the coupling strength is increased continuously.  On the other hand, it also shows that a drop from the coherent state to the incoherent one when the coupling strength is decreased progressively. The critical coupling strengths are different from each other; a larger coupling strength is required to reach the coherent state from the incoherent one. This is a clear indication of the hysteresis phenomena as proven in \cite[Figure 1(b)]{CN11}.
}
%
%
%
%

\section{Conclusion}
In this manuscript we focused our attention to the construction of effective numerical schemes for the solution of the continuum Kuramoto model with diffusion. For such system we introduce a discretization of the phase variable such that nonnegativity of the solution, physical conservations, asymptotic behavior and free energy dissipation are preserved at a discrete level. This approach is then coupled with a suitable collocation method for the frequency variable based on orthogonal polynomials with respect to the given frequency distribution. The method has been tested against the study of phase transitions in terms of the order parameter $r^\infty$ as a function of $K$ for $D>0$. The numerical results agree very well with the theoretical basis, showing that the phase transition is continuous in the identical case and for Gaussian and uniformly coupled noisy oscillators \cite{ABPRS}, and discontinuous for bimodal distributions \cite{BNS,Craw94}. {Hysteresis phenomena are also well-captured in the Kuramoto-Daido type model \cite{Chi, CN11}.} Future research directions and extensions of the present approach will consider the case of the Kuramoto model with inertia \cite{ABPRS,BM}.

%
%
%
%

\section*{Acknowledgments}
JAC was partially supported by the EPSRC grant number EP/P031587/1. YPC was supported by National Research Foundation of Korea (NRF) grant funded by the Korea government (MSIP) (No. 2017R1C1B2012918 and 2017R1A4A1014735) and POSCO Science Fellowship of POSCO TJ Park Foundation. LP acknowledges the support of Imperial College London thanks to the Nelder visiting fellowship. The authors warmly thank Professor Julien Barr\'e and Guy M\'etivier for helpful discussions and valuable comments. 

%
%
%
%
\bibliographystyle{abbrv}
\bibliography{Kuramoto}

\begin{thebibliography}{10}

\bibitem{AB98}
J.~A. Acebr\'on and L.~L. Bonilla.
\newblock Asymptotic description of transients and synchronized states of
  globally coupled oscillators.
\newblock {\em Phys. D}, 114(3-4):296--314, 1998.

\bibitem{ABPRS}
J.~A. Acebr\'on, L.~L. Bonilla, C.~J. P\'erez~Vicente, F.~Ritort, and
  R.~Spigler.
\newblock The {K}uramoto model: A simple paradigm for synchronization
  phenomena.
\newblock {\em Rev. Mod. Phys.}, 77:137--185, Apr 2005.

\bibitem{APS01}
J.~A. Acebr\'on, A.~Perales, and R.~Spigler.
\newblock Bifurcations and global stability of synchronized stationary states
  in the kuramoto model for oscillator populations.
\newblock {\em Phys. Rev. E}, 64:016218, Jun 2001.

\bibitem{Ku}
H.~Araki, editor.
\newblock {\em International {S}ymposium on {M}athematical {P}roblems in
  {T}heoretical {P}hysics}.
\newblock Springer-Verlag, Berlin-New York, 1975.
\newblock Held at Kyoto University, Kyoto, January 23--29, 1975, Lecture Notes
  in Physics, 39.

\bibitem{BD00}
N.~J. Balmforth and R.~Sassi.
\newblock A shocking display of synchrony.
\newblock {\em Phys. D}, 143(1-4):21--55, 2000.

\bibitem{BCCD}
A.~B.~T. Barbaro, J.~A. Ca\~nizo, J.~A. Carrillo, and P.~Degond.
\newblock Phase transitions in a kinetic flocking model of {C}ucker-{S}male
  type.
\newblock {\em Multiscale Model. Simul.}, 14(3):1063--1088, 2016.

\bibitem{BarDeg}
A.~B.~T. Barbaro and P.~Degond.
\newblock Phase transition and diffusion among socially interacting
  self-propelled agents.
\newblock {\em Discrete Contin. Dyn. Syst. Ser. B}, 19(5):1249--1278, 2014.

\bibitem{BCDPZ}
J.~Barr{\'e}, J.~A. Carrillo, P.~Degond, D.~Peurichard, and E.~Zatorska.
\newblock Particle interactions mediated by dynamical networks: Assessment of
  macroscopic descriptions.
\newblock {\em Journal of Nonlinear Science}, 28(1):235--268, Feb 2018.

\bibitem{BDZ}
J.~Barr\'e, P.~Degond, and E.~Zatorska.
\newblock Kinetic theory of particle interactions mediated by dynamical
  networks.
\newblock {\em Multiscale Model. Simul.}, 15(3):1294--1323, 2017.

\bibitem{BM}
J.~Barr\'e and D.~M\'etivier.
\newblock Bifurcations and singularities for coupled oscillators with inertia
  and frustration.
\newblock {\em Phys. Rev. Lett.}, 117:214102, Nov 2016.

\bibitem{BU1}
L.~Basnarkov and V.~Urumov.
\newblock Phase transitions in the {K}uramoto model.
\newblock {\em Phys. Rev. E}, 76:057201, Nov 2007.

\bibitem{BU2}
L.~Basnarkov and V.~Urumov.
\newblock Kuramoto model with asymmetric distribution of natural frequencies.
\newblock {\em Phys. Rev. E}, 78:011113, Jul 2008.

\bibitem{Caglioti1}
D.~Benedetto, E.~Caglioti, and U.~Montemagno.
\newblock On the complete phase synchronization for the {K}uramoto model in the
  mean-field limit.
\newblock {\em Commun. Math. Sci.}, 13(7):1775--1786, 2015.

\bibitem{BF}
M.~Bessemoulin-Chatard and F.~Filbet.
\newblock A finite volume scheme for nonlinear degenerate parabolic equations.
\newblock {\em SIAM J. Sci. Comput.}, 34(5):B559--B583, 2012.

\bibitem{BCC}
F.~Bolley, J.~A. Ca\~nizo, and J.~A. Carrillo.
\newblock Stochastic mean-field limit: non-{L}ipschitz forces and swarming.
\newblock {\em Math. Models Methods Appl. Sci.}, 21(11):2179--2210, 2011.

\bibitem{BNS}
L.~L. Bonilla, J.~C. Neu, and R.~Spigler.
\newblock Nonlinear stability of incoherence and collective synchronization in
  a population of coupled oscillators.
\newblock {\em J. Statist. Phys.}, 67(1-2):313--330, 1992.

\bibitem{BCDS}
C.~Buet, S.~Cordier, and V.~Dos~Santos.
\newblock A conservative and entropy scheme for a simplified model of granular
  media.
\newblock {\em Transport Theory Statist. Phys.}, 33(2):125--155, 2004.

\bibitem{BD}
C.~Buet and S.~Dellacherie.
\newblock On the {C}hang and {C}ooper scheme applied to a linear
  {F}okker-{P}lanck equation.
\newblock {\em Commun. Math. Sci.}, 8(4):1079--1090, 2010.

\bibitem{CCR}
J.~A. Ca\~nizo, J.~A. Carrillo, and J.~Rosado.
\newblock A well-posedness theory in measures for some kinetic models of
  collective motion.
\newblock {\em Math. Models Methods Appl. Sci.}, 21(3):515--539, 2011.

\bibitem{CCH}
J.~A. Carrillo, A.~Chertock, and Y.~Huang.
\newblock A finite-volume method for nonlinear nonlocal equations with a
  gradient flow structure.
\newblock {\em Commun. Comput. Phys.}, 17(1):233--258, 2015.

\bibitem{CCHKK}
J.~A. Carrillo, Y.-P. Choi, S.-Y. Ha, M.-J. Kang, and Y.~Kim.
\newblock Contractivity of transport distances for the kinetic {K}uramoto
  equation.
\newblock {\em J. Stat. Phys.}, 156(2):395--415, 2014.

\bibitem{ChCo}
J.~Chang and G.~Cooper.
\newblock A practical difference scheme for {F}okker-{P}lanck equations.
\newblock {\em Journal of Computational Physics}, 6(1):1--16, 1970.

\bibitem{Chi}
H.~Chiba.
\newblock A proof of the {K}uramoto conjecture for a bifurcation structure of
  the infinite-dimensional {K}uramoto model.
\newblock {\em Ergodic Theory Dynam. Systems}, 35(3):762--834, 2015.

\bibitem{CN11}
H.~Chiba and I.~Nishikawa.
\newblock Center manifold reduction for large populations of globally coupled
  phase oscillators.
\newblock {\em Chaos}, 21(4):043103, 10, 2011.

\bibitem{Craw94}
J.~D. Crawford.
\newblock Amplitude expansions for instabilities in populations of
  globally-coupled oscillators.
\newblock {\em J. Statist. Phys.}, 74(5-6):1047--1084, 1994.

\bibitem{DFL}
P.~Degond, A.~Frouvelle, and J.-G. Liu.
\newblock Phase transitions, hysteresis, and hyperbolicity for self-organized
  alignment dynamics.
\newblock {\em Arch. Ration. Mech. Anal.}, 216(1):63--115, 2015.

\bibitem{DP15}
G.~Dimarco and L.~Pareschi.
\newblock Numerical methods for kinetic equations.
\newblock {\em Acta Numer.}, 23:369--520, 2014.

\bibitem{DB}
F.~D\"orfler and F.~Bullo.
\newblock Synchronization in complex networks of phase oscillators: a survey.
\newblock {\em Automatica J. IFAC}, 50(6):1539--1564, 2014.

\bibitem{Erm}
B.~Ermentrout.
\newblock An adaptive model for synchrony in the firefly pteroptyx malaccae.
\newblock {\em Journal of Mathematical Biology}, 29(6):571--585, Jun 1991.

\bibitem{FGG}
B.~Fernandez, D.~G\'erard-Varet, and G.~Giacomin.
\newblock Landau damping in the {K}uramoto model.
\newblock {\em Ann. Henri Poincar\'e}, 17(7):1793--1823, 2016.

\bibitem{GLP}
G.~Giacomin, E.~Lu\c{c}on, and C.~Poquet.
\newblock Coherence stability and effect of random natural frequencies in
  populations of coupled oscillators.
\newblock {\em J. Dynam. Differential Equations}, 26(2):333--367, 2014.

\bibitem{GPP}
G.~Giacomin, K.~Pakdaman, and X.~Pellegrin.
\newblock Global attractor and asymptotic dynamics in the {K}uramoto model for
  coupled noisy phase oscillators.
\newblock {\em Nonlinearity}, 25(5):1247--1273, 2012.

\bibitem{GP}
S.~N. Gomes and G.~A. Pavliotis.
\newblock Mean {F}ield {L}imits for {I}nteracting {D}iffusions in a
  {T}wo-{S}cale {P}otential.
\newblock {\em J. Nonlinear Sci.}, 28(3):905--941, 2018.

\bibitem{Gosse}
L.~Gosse.
\newblock {\em Computing qualitatively correct approximations of balance laws},
  volume~2 of {\em SIMAI Springer Series}.
\newblock Springer, Milan, 2013.
\newblock Exponential-fit, well-balanced and asymptotic-preserving.

\bibitem{GCR2}
S.~Gupta, A.~Campa, and S.~Ruffo.
\newblock Kuramoto model of synchronization: equilibrium and nonequilibrium
  aspects.
\newblock {\em J. Stat. Mech. Theory Exp.}, (8):R08001, 61, 2014.

\bibitem{GCR1}
S.~Gupta, A.~Campa, and S.~Ruffo.
\newblock Nonequilibrium first-order phase transition in coupled oscillator
  systems with inertia and noise.
\newblock {\em Phys. Rev. E}, 89:022123, Feb 2014.

\bibitem{HX1}
S.-Y. Ha and Q.~Xiao.
\newblock Nonlinear instability of the incoherent state for the
  {K}uramoto-{S}akaguchi-{F}okker-{P}lank equation.
\newblock {\em J. Stat. Phys.}, 160(2):477--496, 2015.

\bibitem{HX2}
S.-Y. Ha and Q.~Xiao.
\newblock Remarks on the nonlinear stability of the {K}uramoto-{S}akaguchi
  equation.
\newblock {\em J. Differential Equations}, 259(6):2430--2457, 2015.

\bibitem{JP1}
S.~Jin and L.~Pareschi.
\newblock Discretization of the multiscale semiconductor {B}oltzmann equation
  by diffusive relaxation schemes.
\newblock {\em J. Comput. Phys.}, 161(1):312--330, 2000.

\bibitem{Lan}
C.~Lancellotti.
\newblock On the {V}lasov limit for systems of nonlinearly coupled oscillators
  without noise.
\newblock {\em Transport Theory Statist. Phys.}, 34(7):523--535, 2005.

\bibitem{LM}
W.~F. Miller and E.~E. Lewis.
\newblock {\em Computational Methods of Neutron Transport}.
\newblock American Nuclear Society, 1993.

\bibitem{Neu}
H.~Neunzert.
\newblock An introduction to the nonlinear {B}oltzmann-{V}lasov equation.
\newblock In {\em Kinetic theories and the {B}oltzmann equation ({M}ontecatini,
  1981)}, volume 1048 of {\em Lecture Notes in Math.}, pages 60--110. Springer,
  Berlin, 1984.

\bibitem{PT2}
L.~Pareschi and G.~Toscani.
\newblock {\em Interacting Multiagent Systems: Kinetic Equations \& Monte Carlo
  Methods}.
\newblock Oxford University Press, 2014.

\bibitem{PZ}
L.~Pareschi and M.~Zanella.
\newblock Structure preserving schemes for nonlinear fokker--planck equations
  and applications.
\newblock {\em Journal of Scientific Computing}, 74(3):1575--1600, Mar 2018.

\bibitem{Paz}
D.~Paz\'o.
\newblock Thermodynamic limit of the first-order phase transition in the
  {K}uramoto model.
\newblock {\em Phys. Rev. E}, 72:046211, Oct 2005.

\bibitem{PR97}
C.~J. Perez and F.~Ritort.
\newblock A moment-based approach to the dynamical solution of the {K}uramoto
  model.
\newblock {\em Journal of Physics A: Mathematical and General}, 30(23):8095,
  1997.

\bibitem{PRK}
A.~Pikovsky, M.~Rosenblum, and J.~Kurths.
\newblock {\em Synchronization}, volume~12 of {\em Cambridge Nonlinear Science
  Series}.
\newblock Cambridge University Press, Cambridge, 2001.

\bibitem{NR}
W.~Press, S.~Teukolsky, W.~Vetterling, and B.~Flannery.
\newblock {\em Numerical {R}ecipes: {T}he {A}rt of {S}cientific {C}omputing,
  {T}hird {E}dition in {C}++}.
\newblock Cambridge University Press, 2007.

\bibitem{RSZ}
C.~Ringhofer, C.~Schmeiser, and A.~Zwirchmayr.
\newblock Moment methods for the semiconductor {B}oltzmann equation on bounded
  position domains.
\newblock {\em SIAM J. Numer. Anal.}, 39(3):1078--1095, 2001.

\bibitem{Saka}
H.~Sakaguchi.
\newblock Cooperative phenomena in coupled oscillator systems under external
  fields.
\newblock {\em Progr. Theoret. Phys.}, 79(1):39--46, 1988.

\bibitem{SG}
D.~L. Scharfetter and H.~K. Gummel.
\newblock Large-signal analysis of a silicon read diode oscillator.
\newblock {\em IEEE Transactions on Electron Devices}, 16(1):64--77, Jan 1969.

\bibitem{Str}
S.~H. Strogatz.
\newblock From {K}uramoto to {C}rawford: exploring the onset of synchronization
  in populations of coupled oscillators.
\newblock {\em Phys. D}, 143(1-4):1--20, 2000.

\bibitem{Win}
A.~T. Winfree.
\newblock Biological rhythms and the behavior of populations of coupled
  oscillators.
\newblock {\em Journal of Theoretical Biology}, 16(1):15 -- 42, 1967.

\bibitem{X}
D.~Xiu.
\newblock {\em Numerical methods for stochastic computations}.
\newblock Princeton University Press, Princeton, NJ, 2010.

\end{thebibliography}

\end{document}